\documentclass{amsart}

\usepackage{geometry}

\usepackage{amssymb}
\usepackage{amsmath}
\usepackage{mathrsfs}
\usepackage{enumitem}
\usepackage{graphicx}
\usepackage{float}
\usepackage{subcaption}
\usepackage{xcolor}

\usepackage[colorlinks=true,
            linkcolor=blue,
            citecolor=blue,
            urlcolor=blue]{hyperref}
            
\theoremstyle{plain}
\newtheorem{theorem}{Theorem}[section]

\newtheorem{proposition}[theorem]{Proposition}

\theoremstyle{definition}
\newtheorem{definition}[theorem]{Definition}
\newtheorem{remark}[theorem]{Remark}

\begin{document}

\title[Inversion formula for the cylindrical Radon transform]
{An analytic inversion formula for the \(n\)-dimensional cylindrical Radon transform}

\author{Jenna Jones and Fatma Terzioglu}
\address{Department of Mathematics, North Carolina State University, Raleigh, NC, USA}
\email[Corresponding author]{jkvarnel@ncsu.edu}
\email{fterzioglu@ncsu.edu}



\begin{abstract}
We study the overdetermined problem of inverting the \(n\)-dimensional cylindrical Radon transform, which maps a function to its integrals over cylindrical surfaces. In practice, the cylindrical Radon transform arises in photoacoustic tomography as a measurement model for integrating line detectors. By exploiting its relationship with the Funk and Radon transforms, we derive an analytic inversion formula for the cylindrical Radon transform. We also present numerical implementations of the proposed inversion algorithm, demonstrating its stability.
\end{abstract}

\maketitle

\section{Introduction}
 The inverse problem of recovering a function from its integrals over families of geometric sets is central to integral geometry and tomography. A classical example is the inversion of the Radon transform, which reconstructs a function from its integrals over hyperplanes and underlies image reconstruction in computed tomography (see, for example, \cite{natterer_mathematics_2001,natterer_mathematical_2001,radon_uber_1917}). In this work, we study the analytic inversion of the \textit{cylindrical Radon transform} (CRT), which maps a function to its integrals over cylindrical surfaces.
The cylindrical Radon transform has applications in photoacoustic tomography, a hybrid imaging technique that combines optical excitation with acoustic detection \cite{kuchment_mathematics_2015,xu_photoacoustic_2006}. In photoacoustic tomography, a short optical pulse illuminates the object being imaged, causing rapid heating and subsequent thermoelastic expansion of the tissue. This expansion generates acoustic pressure waves, a phenomenon known as the photoacoustic effect \cite{tam_applications_1986}. The acoustic waves are measured on an imaging boundary and used to reconstruct the initial pressure distribution. Mathematical models and reconstruction methods for photoacoustic tomography have been studied extensively (see, for example, \cite{kuchment_mathematics_2015,xu_photoacoustic_2006}).

 Classically, point detectors are used in photoacoustic tomography to measure the acoustic waves on the imaging boundary. These measurements correspond to integrals of the initial pressure over spherical surfaces \cite{finch_determining_2004,kuchment_mathematics_2008}. However, point detectors provide only pointwise sampling of the wave field and therefore require dense spatial sampling to achieve stable reconstructions \cite{kuchment_mathematics_2008,xu_photoacoustic_2006}. To address some of these limitations, integrating line detectors have been introduced \cite{burgholzer_temporal_2007,burgholzer_thermoacoustic_2005}. When line detectors are used in circular geometry, the resulting data corresponds to integrals over cylindrical surfaces, leading to the cylindrical Radon transform \cite{haltmeier_frequency_2009,haltmeier_thermoacoustic_2007}. While this geometry offers advantages in reconstruction quality, it also introduces additional challenges in the inversion process \cite{haltmeier_frequency_2009}.

 The cylindrical Radon transform inversion is an overdetermined problem. Therefore, there are multiple inversions that agree under ideal conditions. However, some inversions may amplify noise, while others minimize the effects of noise. In the photoacoustic tomography setting, several geometries enabling inversion of the cylindrical Radon transform have been studied in the literature. Early work by Burgholzer et al. \cite{ burgholzer_temporal_2007,burgholzer_thermoacoustic_2005} developed reconstruction approaches using integrating line detectors and rotation-based geometries. In \cite{burgholzer_thermoacoustic_2005}, the authors studied one- and two-dimensional integrating detectors in thermoacoustic tomography, an imaging technique closely related to photoacoustic tomography, differing primarily in its use of microwave or radio-frequency excitation instead of optical excitation. Their setup required the object to rotate about a fixed axis while the line detectors remained stationary in a fixed plane. In this configuration, the measured data corresponds to the Radon transform of the initial pressure field, allowing the image to be reconstructed using known Radon transform inversion formulas. Later, Burgholzer et al. \cite{burgholzer_temporal_2007} derived an analytic universal backprojection formula for a cylindrical measurement geometry in which line detectors are positioned on a circle surrounding the object. The resulting formula reconstructs the initial pressure field directly from the measured data without requiring iterative reconstruction methods. 

 Limited-view reconstruction problems for thermoacoustic and photoacoustic tomography have been subsequently studied in  \cite{paltauf_experimental_2007,xu_reconstructions_2004}. In \cite{xu_reconstructions_2004}, the authors characterized which image features can be stably recovered from limited-view thermoacoustic tomography data and developed both analytic and algebraic reconstruction methods. Paltauf et al. \cite{paltauf_experimental_2007} investigated limited-view geometries in photoacoustic tomography and showed that such settings lead to incomplete reconstructions, although certain object boundaries can be recovered. The study compared frequency-domain, time-domain direct, and iterative reconstruction algorithms and found that all three methods produced comparable image reconstruction quality. While time-domain methods allow for image-quality refinements through weighting and correction procedures, the frequency-domain algorithms were more computationally efficient. Beyond the development of reconstruction algorithms, authors have also investigated the practical implementation of line-detector photoacoustic imaging systems. Nuster et al. \cite{nuster_photoacoustic_2010} studied photoacoustic imaging with integrating line detectors. They compared detector geometries, such as arc-shaped and box-shaped scanning geometries, and analyzed the corresponding reconstruction procedures. The numerical simulations and experimental studies demonstrated a trade-off between acquisition speed, image quality, and computational complexity.
 
 Haltmeier \cite{haltmeier_inversion_2011} derived exact analytic inversion formulas for the three-dimensional cylindrical Radon transform under a detector configuration in which all cylinder axes are orthogonal to the north pole. The reconstruction method proceeds through an integral relation involving the X-ray transform and the circular Radon transform. By applying known inversion formulas for these transforms, explicit reconstruction formulas for the cylindrical Radon transform are obtained.

In this work, we study the analytic inversion of the cylindrical Radon transform associated with an overdetermined family of cylindrical surfaces in \(\mathbb{R}^n\). We analyze its geometric properties and establish an integral identity relating the cylindrical Radon transform to the Radon and Funk transforms. This identity allows us to derive an analytic inversion method for the cylindrical Radon transform using known inversion formulas for the Funk and Radon transforms. We then demonstrate the stability of the proposed inversion method through numerical experiments based on a ball phantom.

The paper is organized as follows. In Section~\ref{CRT+prelims}, we formulate the cylindrical Radon transform, discuss its basic properties, and recall two related integral transforms together with their known inversion formulas. In Section~\ref{analyticalinversion}, we present our main theoretical results, including the analytic inversion method. Numerical reconstructions and an error analysis are provided in Section~\ref{numerical-results}. Finally, Section~\ref{conclusion} concludes the paper with a discussion of the results and directions for future work.

\section{Preliminaries}\label{CRT+prelims}
We begin by defining the cylindrical Radon transform and discussing some of its geometric properties, including the imaging domain and a parameterization of cylindrical surfaces in \(\mathbb{R}^n\). We then review the related Funk and Radon transforms, together with their associated inversion formulas.

\subsection{Definition of the Cylindrical Radon Transform}
Throughout the paper, $\mathbb{S}^{n-1}$ denotes the unit sphere in $\mathbb{R}^n$, and $|\mathbb{S}^{n-1}|$ denotes its $(n-1)$-dimensional surface measure.

We can define a cylinder in \(\mathbb{R}^n\) with radius \(r>0\) and centered around the axis line 
\[ \{p+tv: t\in \mathbb{R}\},\quad
 v\in \mathbb{S}^{n-1},\quad  p\in \mathbb{R}^n,\]
 by the set 
 \[\left\{y\in \mathbb{R}^n: |y-p|^2-[v\cdot(y-p)]^2=r^2\right\},\]
which consists of all points \(y\) at Euclidean distance \(r\) from the line through $p$ in the direction $v$. 

Let \(\mathcal{C}_c^\infty(\mathbb{R}^n)\) denote the space of smooth compactly supported functions on \(\mathbb{R}^n\). Throughout the paper we fix \(a>0\), and consider \(f\in C_c^\infty(\mathbb{R}^n)\) with \({\rm supp}\,f\subset B(0,a)\), the \(n\)-dimensional ball of radius \(a\) and centered at the origin. Since \(f\) vanishes outside of \(B(0,a)\), it suffices to consider cylinders whose axes intersect the sphere of radius \(a\) centered at the origin, denoted \(\mathbb{S}_a^{n-1}\). Thus we choose to restrict \(p\in \mathbb{S}_a^{n-1},\) and consider the cylinders
\begin{equation}\label{Cylinderset}\mathscr{C}(v,p,r):=\left\{y\in \mathbb{R}^n: |y-p|^2-[v\cdot (y-p)]^2=r^2\right\},\end{equation}
with parameters \((v,p,r)\) in the space \[Z:=\mathbb{S}^{n-1}\times \mathbb{S}_a^{n-1}\times (0,\infty),\] (See Figure \ref{fig: Cylinder visual}).
\begin{figure}[htbp]
\centering\includegraphics[width=.35\linewidth]{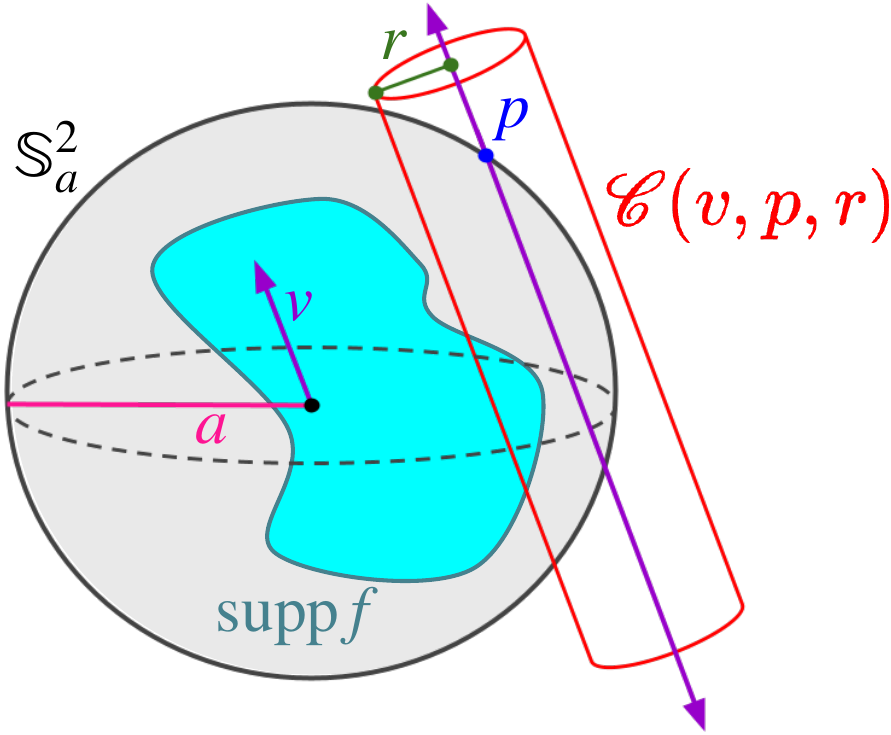}
        \caption{The imaging domain \(B(0,a)\) and a cylinder \(\mathscr{C}(v,p,r)\). }
        \label{fig: Cylinder visual}
    \end{figure}
    
\begin{definition}[The Cylindrical Radon Transform]\label{cylindrical Radon transform}
Let $f \in \mathcal{C}_c^\infty(\mathbb{R}^n)$ with ${\rm supp}\,f \subset B(0,a)$, for some $a>0$. The cylindrical Radon transform $Cf$ of $f$ is a function \(Cf:Z\to \mathbb{R}\) defined by \[Cf(v,p,r)=\int_{\mathscr{C}(v,p,r)} f(y) \,dS(y).\]
\end{definition}
Since \({\rm supp}\,f\subset B(0,a)\), for any \(v\in\mathbb{S}^{n-1}\) and \(p\in\mathbb{S}_a^{n-1}\), \(Cf(v,p,r)=0\) whenever \(r>2a\). Equivalently, we can write the cylindrical Radon transform using the one-dimensional Dirac delta distribution as follows.
\begin{proposition}\label{CRTdefdelta}
Let $f \in \mathcal{C}_c^\infty(\mathbb{R}^n)$ with ${\rm supp}\,f \subset B(0,a)$, for some $a>0$. Then,
 \begin{equation}\label{eq:CRTdefdelta}
    Cf(v,p,r)=2r\int_{\mathbb{R}^n}
f(y)\,
\delta\!\bigl(|y-p|^2-[v\cdot (y-p)]^2-r^2\bigr)\,
\,
dy. \end{equation}
\end{proposition}

\begin{proof}
Using the level set function of a cylinder, $$\Phi_{v,p,r}(y) \;:=\; |y-p|^2 - [v\cdot (y-p)]^2 - r^2, 
\qquad y\in \mathbb{R}^n,$$ 
we can write
\begin{align*}
    Cf(v,p,r)=\int_{\mathscr{C}(v,p,r)} f(y)\, d\sigma(y)
=
\int_{\mathbb{R}^n} f(y)\, \delta\!\bigl(\Phi_{v,p,r}(y)\bigr)\, \bigl|\nabla_y \Phi_{v,p,r}(y)\bigr|\,dy.
\end{align*}
Fixing $v,p,r$ and differentiating with respect to $y$, we obtain
\[
\nabla_y \Phi_{v,p,r}(y)
=
2(y-p) - 2[v\cdot (y-p)]\,v
=
2\bigl(y-p-[v\cdot (y-p)]v\bigr).
\]
Therefore
\[
\bigl|\nabla_y \Phi_{v,p,r}(y)\bigr|
=
2\bigl|y-p-[v\cdot (y-p)]v\bigr|
=
2\sqrt{|y-p|^2-[v\cdot (y-p)]^2}.
\]
Substituting $\Phi_{v,p,r}$ and its gradient magnitude gives
\begin{align*}
Cf(v,p,r)
&=
\int_{\mathbb{R}^n}
f(y)\,
\delta\!\bigl(|y-p|^2-[v\cdot (y-p)]^2-r^2\bigr)\,
2\sqrt{|y-p|^2-[v\cdot (y-p)]^2}\,
dy\\
&=2r\int_{\mathbb{R}^n}
f(y)\,
\delta\!\bigl(|y-p|^2-[v\cdot (y-p)]^2-r^2\bigr)\,
\,
dy.
\end{align*}
Here the second equality follows because the delta distribution restricts the integral to the cylinder $\mathscr{C}(v,p,r)$, where
 $|y-p|^2-[v\cdot (y-p)]^2=r^2,$ and hence the weight reduces to \(2r\).
\end{proof}

In the special case \(v=e_n=(0,\ldots,0,1)\in \mathbb{R}^n\), we write \[p=(\bar{p},p_n),\quad y=(\bar{y},y_n).\] 

Then, \begin{equation*}\mathscr{C}(e_n,p,r):=\left\{y\in \mathbb{R}^n: |\bar{y}-\bar{p}|^2=r^2\right\}.\end{equation*} Accordingly, the cylindrical Radon transform becomes \begin{equation*}
    Cf(e_n,p,r)= 2r\int_\mathbb{R}\int_{\mathbb{R}^{n-1}} f(\bar{y},y_n)\,\delta(|\bar{y}-\bar{p}|^2-r^2)\,d\bar{y}\, dy_n.\end{equation*}
   To further simplify this integral, we introduce cylindrical coordinates centered at \(\bar{p}\) yielding 
 \[\bar{y}=\bar{p}+\rho\omega, \quad y_n=z,  \quad \rho>0, \quad \omega\in \mathbb{S}^{n-2}.\] Using a well-known Dirac delta identity, \[\delta(|\bar{y}-\bar{p}|^2-r^2)=\delta(\rho^2-r^2)=\frac{1}{2r}\delta(\rho-r),\] we obtain
 \begin{align*}
    Cf(e_n,p,r)&=\int_{\mathbb{R}} \int_{\mathbb{S}^{n-2}}\int_0^\infty f(\bar{p}+\rho\omega,z)\delta(\rho-r)\,\rho^{n-2}\,d\rho\, d\omega\, dz.
    \end{align*}
    Since \(r>0,\)  evaluating the integral with respect to \(\rho\) gives
    \begin{align}\label{vCRT}
     Cf(e_n,p,r)&=\int_{\mathbb{R}} \int_{\mathbb{S}^{n-2}} f(\bar{p}+r\omega,z)\,r^{n-2}\, d\omega\, dz.
    \end{align}
We call this integral the \textit{vertical cylindrical Radon transform}.

\subsection{Geometric Properties of the Cylindrical Radon Transform} Next, we present some geometric properties of the cylindrical Radon transform that will be useful in the sequel. 

\begin{proposition}[Properties of the Cylindrical Radon Transform]\label{prop:properties}
Let \(f \in \mathcal{C}_c^\infty(\mathbb{R}^n)\) with ${\rm supp}\,f \subset B(0,a)$, for some $a>0$, and let \(C f(v,p,r)\) denote the cylindrical Radon transform of \(f\), where \((v,p, r )\in Z\). Then, the following properties hold:
\begin{enumerate}[label=(\roman*)]

    \item \textbf{Evenness with respect to v:}
    \begin{equation}\label{Cf even}
    C f(v,p, r) = C f(-v,p, r).
    \end{equation}
    
    \item \textbf{Reflection symmetry:}
    \begin{equation*}
    C f(v,p, r) = C( f)(v,p-2(p\cdot v)v, r).
    \end{equation*}

    \item \textbf{Translation invariance:} Define \(T_cf(x)=f(x+c).\) Then \[C(T_cf)(v,p,r)=Cf(v,p+c,r).\] 
    
     \item \textbf{Rotation invariance:} For any \(A \in SO(n)\), the set of orthogonal \(n\times n\) matrices with determinant one, define \(M_A f(x) := f(Ax)\). Then
   \begin{equation}\label{M_ACf=CM_Af}
    C f(Av, Ap, r) = C(M_A f)(v,p, r).
    \end{equation}

    \item \textbf{Scaling invariance:} For  \(s >0\), let \(D_s f(x) := f(sx)\). Then
    \[
    C f(v,sp, sr) = s^{n-2} C(D_s f)(v,p, r).
    \]

\end{enumerate}
\end{proposition}
\begin{proof}
    \medskip
    \noindent\textbf{(i).} The evenness with respect to \(v\) follows directly from the fact that \(v\) and \(-v\) define the same central axis, and thus the same cylinder. 
    
    \medskip
    \noindent\textbf{(ii).} Property (ii) holds since the point \(p-2(p\cdot v) v\) is the other intersection point of the axis line with the sphere \(\mathbb{S}^{n-1}_a\). Indeed, let \[r(t):=p-tv,\] which parameterizes the cylinder axis. The intersection points of this line with the sphere \(\mathbb{S}^{n-1}_a\) satisfy
    \begin{equation*}
        (p-tv)\cdot (p-tv)=a^2.
    \end{equation*} Simplifying and using \(p\cdot p=a^2\) gives 
    \begin{equation*}
         t(t-2p\cdot v)=0.
    \end{equation*} The solution occurring at \(t=0\) corresponds to the point \(p\). The second intersection occurs at \(t_0=2p\cdot v.\) Thus, \begin{equation*}
        r(t_0)=p-2(p\cdot v)v,
    \end{equation*} gives the second point where the line \(r(t)=p-tv\) intersects the sphere \(\mathbb{S}_a\). Hence, both parameters \((v,p,r)\) and \((v,p-2(p\cdot v) v,r)\) define the same cylinder, which implies (ii).
    
    \medskip
    \noindent\textbf{(iii).} Let \(c\in \mathbb{R}^n.\) Then, \begin{equation}\label{eq: translation}
C(T_cf)(v,p,r)=2r\int_{\mathbb{R}^n} f(x+c)\delta(|x-p|^2-[v\cdot (x-p)]^2-r^2)\, dx.\end{equation} Setting \(z=x+c\) gives \[C(T_cf)(v,p,r)=2r\int_{\mathbb{R}^n} f(z)\delta(|z-(p+c)|^2-[v\cdot (z-(p+c))]^2-r^2)\, dz=Cf(v,p+c,r).\]

\medskip
    \noindent\textbf{(iv).}
     Let \(A\in SO(n)\). Using the change of variable \(y=Ax\), we get \(dy=dx\) since \(\det(A)=1\). Then \begin{align*}
    Cf(Av,Ap,r)
    &=2r\int_{\mathbb{R}^n} f(y)\,\delta( |y-Ap|^2-[Av\cdot (y-Ap)]^2-r^2)\, dy\\
     &=2r\int_{\mathbb{R}^n} f(Ax)\,\delta( |A(x-p)|^2-[Av\cdot A(x-p)]^2-r^2)\, dx. \end{align*}
    Since \(A\) is a rotation matrix, we have 
    \[Av\cdot A(x-p)=(Av)^TA(x-p)=v^TA^TA(x-p)=v\cdot (x-p),\] and \[|A(x-p)|=|x-p|.\] Thus,
     \begin{align*}
      Cf(Av,Ap,r) &=2r\int_{\mathbb{R}^n} f(Ax)\delta(|x-p|^2-[v\cdot (x-p)]^2-r^2) dx\\
    &=  C(M_Af)(v,p,r).
\end{align*}
 \medskip
    \noindent\textbf{(v).}
    Let \(s>0\). We compute \begin{align*}
    Cf(v,sp,sr)
     &=2r\int_{\mathbb{R}^n} f(y)\,\delta( |y-sp|^2-[v\cdot (y-sp)]^2-s^2r^2)\, dy.
     \end{align*}
     Let \(y=sx\) and consequently, \(dy=s^n dx\). Then
     \begin{align*}
   Cf(v,sp,sr)
    &=2r\int_{\mathbb{R}^n} f(sx)\delta\Big(s^2 \big(|x-p|^2-\big[v\cdot (x-p)]^2-r^2\big)\Big)\, s^n\, dx
    \end{align*}
    We recall the scaling property of the Dirac delta distribution that states \(\delta(at)=\frac{\delta(t)}{a}\) for any scalar \(a>0\). Hence, 
     \begin{align*}
    Cf(v,sp,sr)
    &=2r\int_{\mathbb{R}^n} f(sx)\frac{\delta( |(x-p)|^2-[v\cdot (x-p)]^2-r^2)}{s^2} s^n dx\\
    &=2s^{n-2}r\int_{\mathbb{R}^n} f(sx)\delta( |x-p|^2-[v\cdot (x-p)]^2-r^2) dx\\
    &=s^{n-2}C(D_sf)(v,p,r).
\end{align*}

\end{proof}

\subsection{The Cylindrical Radon Transform of the Characteristic Function of a Ball in Three-dimensions}
In this section, we present an analytical formula for the cylindrical Radon transform of the characteristic function of an $n$-dimensional ball. This formula will be used for simulating forward data for a ball phantom in section \ref{numerical-results}. 
\begin{proposition}
\label{prop:ballCRT}
For $c\in\mathbb{R}^3$ and $t>0$, let
\[
f(x) = \chi_{B(c,t)}(x)
=
\begin{cases}
1, & |x-c| \le t,\\
0, & |x-c| > t.
\end{cases}
\]
 
 For \((v,p,r)\in Z\), define 
 \[
 d := |\mathrm{proj}_{v^\perp}(p-c)|.
 \]
 
 Then,  
 \begin{equation}\label{eq:ball CRT}
Cf(v,p,r)
=
\begin{cases}
\displaystyle
\int_0^{2\pi}
2r\sqrt{t^2 - d^2 - r^2 - 2dr\cos\theta}\, d\theta,
& \text{if } \max\{0,d-t\}\le r \le d+t,\\[12pt]
0, & \text{otherwise}.
\end{cases}
\end{equation}
\end{proposition}

\begin{proof}
The cylindrical Radon transform of \(\mathcal{X}_{B(c,t)}\) is the surface area of the cylinder inside the ball \(B(c,t)\). By Proposition \ref{prop:properties} (iii), \[Cf(v,p,r)=C\mathcal{X}_{B(c,t)}(v,p,r)=C\mathcal{X}_{B(0,t)}(v,p-c,r).\] Let \(q:=p-c\), and \(A\in SO(3)\) such that $Av=e_3$. Since the ball \(B(0,t)\) is rotationally invariant, by Proposition \ref{prop:properties} (iv), we have 
\[
C\mathcal{X}_{B(0,t)}(v,q,r)=C\mathcal{X}_{B(0,t)}(e_3,Aq,r).
\]
By the rotation invariance on the \(xy\)-plane, we further have  \[C\mathcal{X}_{B(0,t)}(e_3,Aq,r)=C\mathcal{X}_{B(0,t)}(e_3,(d,0,0),r),\] where
\[
d = |\mathrm{proj}_{e_3^\perp} Aq| =|\mathrm{proj}_{v^\perp} (p-c)|.
\] 
Therefore, by the definition of the vertical cylindrical Radon transform \eqref{vCRT},
\begin{equation}\label{doubleCfball}
Cf(v,p,r)
=C\mathcal{X}_{B(0,t)}(e_3,(d,0,0),r)=
\int_0^{2\pi}
\int_{\mathbb{R}}
\mathcal{X}_{B(0,t)}(d + r\cos\theta, r\sin\theta, z)\, dz\, r\, d\theta.
\end{equation}

The integrand is non-vanishing if and only if
\[
(d + r\cos\theta)^2 + r^2\sin^2\theta + z^2 \le t^2,
\]
or equivalently
\[
z^2 \le t^2 - d^2 - r^2 - 2dr\cos\theta.
\]
Hence, for each angle $\theta$ such that the right-hand side is nonnegative,
\[
z \in \left[-\sqrt{t^2 - d^2 - r^2 - 2dr\cos\theta},
\sqrt{t^2 - d^2 - r^2 - 2dr\cos\theta}\right].
\]
For a fixed angle $\theta$, the integrand in \eqref{doubleCfball} equals $1$ precisely when
\[
|z|
\le
\sqrt{t^2-d^2-r^2-2dr\cos\theta},
\]
provided that
\[
t^2-d^2-r^2-2dr\cos\theta \ge 0.
\]
Hence the inner integral is the length of the interval of admissible
$z$-values:
\begin{align*}
\int_{\mathbb{R}}&
\mathcal{X}_{B(0,t)}(d+r\cos\theta,r\sin\theta,z)\,dz \\
&=
2\sqrt{t^2-d^2-r^2-2dr\cos\theta}\,
\mathcal{X}_{[0,\infty)}
\!\left(t^2-d^2-r^2-2dr\cos\theta\right).
\end{align*}
The square-root is real for at least one $\theta$ if and only if \[t^2 - d^2 - r^2 - 2dr\cos\theta\geq 0,\] for some \(\theta\in[0,2\pi).\) The left-hand side is maximized when \(\cos\theta=-1.\) Thus, the condition for intersection becomes \[t^2-(d-r)^2\geq 0,\] or equivalently,
\[
|d - r| \leq t.
\]
Since \(r\geq 0\), this yields the condition
\[
\max\{0,d-t\} \le r \le d+t.
\]
In the case where this condition fails, the cylinder does not intersect the ball, so the transform is zero. Substituting into 
\eqref{doubleCfball} yields
\begin{align}
Cf(v,p,r)
&=
\int_0^{2\pi}
2r\sqrt{t^2 - d^2 - r^2 - 2dr\cos\theta}\,
\mathcal{X}_{[0,\infty)}(t^2 - d^2 - r^2 - 2dr\cos\theta)\, d\theta\\[6pt]
&=\begin{cases}
\displaystyle
\int_0^{2\pi}
2r\sqrt{t^2 - d^2 - r^2 - 2dr\cos\theta}\, d\theta,
& \text{if } \max\{0,d-t\}\le r \le d+t,\\[10pt]
0, & \text{otherwise}.\notag
\end{cases}
\end{align}
\end{proof}
We note that \eqref{eq:ball CRT} is an elliptic integral, so it cannot be simplified further.

\subsection{Some Auxiliary Integral Transforms}
In this section, we recall the Radon and Funk transforms and their inversions which will be useful in developing our inversion method for the cylindrical Radon transform. We start with the Radon transform, which maps a function to its integrals over hyperplanes (see, for example, \cite{natterer_mathematics_2001,radon_uber_1917}).
\begin{definition}[The Radon Transform]
    For \(f\in\mathcal{S}(\mathbb{R}^n)\), the Schwartz space of rapidly decreasing smooth functions, the Radon transform \(Rf:\mathbb{S}^{n-1}\times \mathbb{R} \to \mathbb{R}\) is given by
\begin{equation}\label{Radon transform}
    Rf(\omega,s)=\int_{x\cdot \omega =s} f(x)dx=\int_{\mathbb{R}^n} f(x)\delta (x\cdot \omega-s) dx,
\end{equation}
where \(\omega\in \mathbb{S}^{n-1}\) denotes the normal direction of the hyperplane and \(s\in \mathbb{R}\) is its signed distance from the origin. 
\end{definition}
 The Radon transform is even: 
\begin{equation}\label{Rfeven}
Rf(-\omega,-s)=Rf(\omega,s),
\end{equation}
and satisfies the following translation and rotation invariance properties:
\begin{align}
R(T_p f)(\omega,s) &= Rf(\omega,s+\omega\cdot p), \label{R(T_pf)=T_pRf}\\
R(M_A f)(\omega,s) &= Rf(A\omega,s):=M_ARf(\omega,s), \label{M_ARf=RM_Af}
\end{align}
where $T_p f(x)=f(x+p)$ and $M_A f(x)=f(Ax)$ for $A \in SO(n)$. 

In some arguments, we will use the notations \begin{align*}
R_sf(\omega):=Rf(\omega,s),\quad \quad R_\omega f(s):=Rf(\omega,s), 
\end{align*}  where the subscript indicates the variable being held fixed.

We now turn to the inversion of \(n\)-dimensional Radon transform which is as follows \cite{natterer_mathematics_2001}: 
\begin{equation}\label{ndimradoninversion}
f(x) = \frac{(-1)^{(n-1)/2}}{2(2\pi)^{n-1}} 
\begin{cases}
    \int_{\mathbb{S}^{n-1}}(Rf)^{(n-1)}(\omega,x\cdot \omega) d\omega,\quad \text{odd}\,\, n\\
    (-1)^{-1/2}\int_{\mathbb{S}^{n-1}}\mathcal{H}(Rf)^{(n-1)}(\omega,x\cdot \omega) d\omega, \quad \text{even}\,\, n,
\end{cases}
\end{equation}
where  \[\mathcal{H}g(t):=\frac{1}{\pi}p.v.\int_{\mathbb{R}}\frac{g(s)}{t-s}\, ds,\quad \text{and} \quad (Rf)^{(n-1)}(\omega,s):=\frac{\partial^{(n-1)}}{\partial s^{(n-1)} }R(\omega,s).\]

In practical applications, the inversion formula (\ref{ndimradoninversion}) is often implemented via the filtered backprojection algorithm \cite{natterer_mathematical_2001}.

We will also use the Funk transform, which integrates a function over \((n-2)\)-dimensional subspheres of the unit sphere \(\mathbb{S}^{n-1}\subset \mathbb{R}^n\), lying in hyperplanes through the origin (see, for example, \cite{funk_uber_1913,helgason_integral_2011,rubin_introduction_2015}). These subspheres are commonly referred to as great subspheres or great circles in the case that \(n=3\). 

\begin{definition}[The Funk Transform]
    For \(f\in \mathcal{C}^\infty(\mathbb{S}^{n-1})\), the Funk Transform, \(Ff:\mathbb{S}^{n-1}\to \mathbb{R}\) is defined by 
\[
    F g(v)=\int_{\mathbb{S}^{n-1}} g(\omega)\delta(\omega\cdot v)\,d\omega.
\] 
\end{definition}
 The Funk transform satisfies the rotational invariance property:
 \begin{equation}\label{FM_Af=M_AFf}
     F(M_Ag)(v)=Fg(A v),
\end{equation} 
where \(M_A g(\omega):=g(A\omega)\) with \(A\in SO(n)\). 

We define the averaged dual Funk transform of a function \( g \in C^\infty(\mathbb{S}^{n-1}) \) at a point \( \omega \in \mathbb{S}^{n-1} \) and angular distance \( \gamma \in [0, \pi/2]\)
as the average of the Funk transforms of \( g \) over all great circles lying at a distance \( \gamma \) from the point \( \omega  \) and denote it as \(
\check{F}g_\gamma(\omega)
\). More precisely, 
\begin{equation}\label{averaged_dual_funk}
\check{F}g_\gamma(\omega) = \int_{d(\omega, \xi) = \gamma} Fg(\xi) \, d\mu(\xi),
\end{equation}
where the integral is taken over all great circles \( \xi \subset \mathbb{S}^{n-1} \) at angular distance \( p \) from the point \( \omega \), and \( d\mu \) denotes the normalized, invariant measure on this family of great subspheres. 

An inversion formula by Helgason \cite{helgason_integral_2011} states that if 
 \(g\in C_{even}^\infty(\mathbb{S}^{n-1})\), then for \(\omega\in \mathbb{S}^{n-1}\), 
 \begin{equation}\label{ndimFTI}
     g(\omega)=\frac{2^{n-2}}{(n-3)!|\mathbb{S}^{n-1}|}
\Bigg[
\Big(\frac {d}{d(t^2)}\Big)^{n-2}
\int_{0}^{t}\check{F}g_{\cos ^{-1}q}(\omega )
q^{n-2}(t^{2}-q^{2})^{\frac{n}{2}-2} \, dq
\Bigg]_{t=1}.\end{equation}

\section{Inversion of the Cylindrical Radon Transform}\label{analyticalinversion}
 In this section, we study the inversion of the overdetermined cylindrical Radon transform associated with the $(2n-1)$-parameter family of cylinders defined in \eqref{Cylinderset}. Our approach is based on an integral identity that relates cylindrical averages to the classical Radon and Funk transforms. This identity reduces the inversion problem to weighted radial integration and the application of the known inversion formulas for the Funk and Radon transforms.

\begin{proposition}\label{Prop:Cf=RFf}
   Let \(f\in \mathcal{C}_c^\infty(\mathbb{R}^n)\) such that \({\rm supp}\,f\subset B(0,a)\). For all \(r>0\) and \(s\geq 0\), define \[g(s,r)=|\mathbb{S}^{n-3}|\frac{(r^2-s^2)^{(n-4)/2}}{r^{n-3}}.\] Then, 
    \begin{equation}\label{propeq}
       \int_{s}^\infty Cf(v,p,r)g(s,r)dr=\int_{\mathbb{S}^{n-1}} Rf(\sigma, s+\sigma\cdot p)\delta (\sigma \cdot v)d\sigma =F(R_{s}T_{p}f)(v). \end{equation}
\end{proposition}
\begin{proof}
     We first prove the proposition for the vertical cylindrical Radon transform with the cylinder axis passing through the point \(q=(\bar{q},q_n)\in \mathbb{S}_a^{n-1}\). We then extend the result to the general case using the invariance properties of the cylindrical Radon, Funk, and the classical Radon transforms. By definition of the vertical cylindrical Radon transform \eqref{vCRT}, we have
    \begin{align*}
    \int_{s}^\infty Cf(e_n,q,r)\,g(s,r)\, dr
    &=\int_{s}^\infty  \int_{-\infty}^\infty \int_{\mathbb{S}^{n-2}} f(\bar{q}+r \omega,z)\, r^{n-2}\, d\omega\, dz\, g(s,r)\, dr\\
    &=|\mathbb{S}^{n-3}|\int_{s}^\infty  \int_{-\infty}^\infty \int_{\mathbb{S}^{n-2}} f(\bar{q}+r \omega,z)\,(r^2-s^2)^{(n-4)/2}\,r\,d\omega \,dz\, dr. \end{align*}
    Applying Fubini's Theorem, and using \(f_z(\Bar{x}):=f(\Bar{x},z)\), \((\Bar{x},z)\in \mathbb{R}^{n-1}\times \mathbb{R}\),
    \begin{align*}
    \int_{s}^\infty Cf(e_n,q,r)\,g(s,r)\, dr
    &=  |\mathbb{S}^{n-3}|\int_{-\infty}^\infty \int_{\mathbb{S}^{n-2}} \int_{s}^\infty f_z(\bar{q}+r \omega)\,(r^2-s^2)^{(n-4)/2}\,r\, dr\, d\omega\, dz.
    \end{align*}
    Using the identity (See \cite{terzioglu_inversion_2015}, Lemma 21),
    \[|\mathbb{S}^{n-3}| \int_{\mathbb{S}^{n-2}}\int_{s}^\infty f_z(\bar{q}+r\omega)\,(r^2-s^2)^{(n-4)/2}\,r\,dr\,d\omega=\int_{\mathbb{S}^{n-2}} Rf_z(\omega,s+\bar{q}\cdot \omega) \,d\omega,\]
    we obtain
    \begin{align*}
    \int_{s}^\infty Cf(e_n,q,r)\,g(s,r)\,dr
    &=  \int_{-\infty}^\infty \int_{\mathbb{S}^{n-2}}  Rf_z(\omega,s+\bar{q}\cdot \omega)\,d\omega \,dz\\
    &=  \int_{-\infty}^\infty \int_{\mathbb{S}^{n-2}}  \int_{\mathbb{R}^{n-1} }f(\bar{x},z)\, \delta (\bar{x} \cdot \omega-s-\bar{q}\cdot \omega)\,d\bar{x}\,d\omega \,dz,
    \end{align*}
    by the definition of the Radon transform. Note that \(\bar{x} \cdot \omega=x\cdot(\omega,0)\) for \(x=(\bar{x},z)\in \mathbb{R}^{n-1}\times \mathbb{R}\) and similarly, \(\bar{q}\cdot \omega=q\cdot(\omega,0)\). Hence, by switching the order of integration once more, we obtain
    \[ \int_{s}^\infty Cf(e_n,q,r)g(s,r)dr  
     =   \int_{\mathbb{S}^{n-2}}  Rf((\omega,0),s+q\cdot (\omega,0))d\omega= \int_{\mathbb{S}^{n-2}}  R(T_{q}f)((\omega,0),s)d\omega, 
    \]
    by translation invariance of the Radon transform \eqref{R(T_pf)=T_pRf}.
    
    We now express this identity in terms of the Funk transform. Using the Dirac-delta distribution, we write
    \begin{align*}
    \int_{s}^\infty Cf(e_n,q,r)\,g(s,r)\,dr  
    &=  \int_{\mathbb{S}^{n-2}}  R_s(T_{q}f)((\omega,0))\,d\omega\\
    &=\int_{-1}^1\int_{\mathbb{S}^{n-2}} R_{s} (T_{q}f)((\sqrt{1-\sigma_n^2} \omega, \sigma_n))\,\delta(\sigma_n)\,(1-\sigma_n^2)^\frac{n-3}{2}\,d\omega \,d\sigma_n.
    \end{align*} 
    Let \(\sigma=(\sqrt{1-\sigma_n^2} \omega, \sigma_n)\in \mathbb{S}^{n-1}.\) Then, \(d\sigma=(1-\sigma_n^2)^{\frac{n-3}{2}}\, d\omega\, d\sigma_n,\) and thus
    \begin{equation}\label{propeqen}
     \int_{s}^\infty Cf(e_n,q,r)\,g(s,r)\,dr  
    =\int_{\mathbb{S}^{n-1}} R_s(T_{q}f)(\sigma)\,\delta(\sigma\cdot e_n)\,d\sigma=F(R_sT_{q}f)(e_n),
    \end{equation}
    which yields \eqref{propeq} when \(v=e_n\). 
    
    We now generalize this identity using the invariance properties of the cylindrical Radon, Funk, and the classical Radon transforms.  Let \(v\in \mathbb{S}^{n-1}\) be arbitrary. There exists a rotation matrix \(A\in SO(n)\) such that \(Ae_n=v\). We take \(q:=A^{-1}p\). Using the rotational invariance of the cylindrical Radon transform (Proposition \ref{prop:properties} (iv)),
     we write \[Cf(v,p,r)=Cf(Ae_n,Aq,r)=C(M_A f)(e_n,q,r).\] Then, equation \eqref{propeqen} implies that 
     \begin{align*}
    \int_{s}^\infty Cf(v,p,r)g(s,r)dr  &= \int_{s}^\infty C(M_Af)(e_n,q,r)g(s,r)dr=F(R_sT_{q}M_Af)(e_n).
    \end{align*}
    Observing that
    \begin{align*}
    T_{q}(M_Af)(x)&=M_Af(x+q)=f(Ax+Aq)=T_{Aq}f(Ax)=M_A(T_{p}f)(x),
    \end{align*}  
    and using the definition of the Funk transform, we obtain
    \begin{align*}
    \int_{s}^\infty Cf(v,p,r)g(s,r)dr &=F(R_sM_AT_{p}f)(e_n)\\
    &=\int_{\mathbb{S}^{n-1}} R(M_AT_{p}f)(\omega, s) \delta(\omega\cdot e_n)\, d\omega\\
    &=\int_{\mathbb{S}^{n-1}} R(T_{p}f)(A\omega, s) \delta(\omega\cdot e_n)\, d\omega,
    \end{align*} 
    where we have used the rotational invariance of the Radon transform \eqref{M_ARf=RM_Af}.
    Changing variables by letting \(\sigma=A\omega\), implies
    \begin{align*}
    \int_{s}^\infty Cf(v,p,r)g(s,r)dr 
    &=\int_{\mathbb{S}^{n-1}} R(T_{p}f)(\sigma, s) \delta(A^{-1}\sigma\cdot e_n)\, d\sigma\\
    &=\int_{\mathbb{S}^{n-1}} R(T_{p}f)(\sigma, s) \delta(\sigma\cdot Ae_n)\, d\sigma\\
    &=\int_{\mathbb{S}^{n-1}} R(T_{p}f)(\sigma, s) \delta(\sigma\cdot v)\, d\sigma,
    \end{align*} 
    where we have used rotational invariance of the Lebesgue measure on \(\mathbb{S}^{n-1}\) and the dot product. Finally, by the definition of the Funk transform, 
    \begin{align*}
    \int_{s}^\infty Cf(v,p,r)g(s,r)dr 
    &=\int_{\mathbb{S}^{n-1}} R(T_{p}f)(\sigma, s) \delta(\sigma\cdot v)\, d\sigma=F(R_sT_{p}f)(v),
    \end{align*}
    which completes the proof.           
\end{proof} 
The identity \eqref{propeq} expresses the weighted averages of the cylindrical Radon transform with respect to the radial variable as a composition of the Radon and Funk transforms. This decomposition is the key step in deriving an inversion formula for the cylindrical Radon transform. We first perform the weighted integration with respect to \(r\) on the left hand side of Equation \eqref{propeq}. Then we apply the \(n\)-dimensional Funk transform inversion \eqref{ndimFTI}. Lastly, we apply the \(n\)-dimensional Radon transform inversion \eqref{ndimradoninversion} to reconstruct the function \(f\).
\begin{theorem}[Inversion of the Cylindrical Radon Transform]
Let \(f\in C_c^\infty(\mathbb{R}^n)\) such that \(\operatorname{supp}\, f\subset B(0,a)\) for some \(a>0.\) Let \(g(r,s) \) be as defined in Proposition \ref{Prop:Cf=RFf}, and \(R^{-1}\) and \(F^{-1}_{v\rightarrow w}\) formally denote the Radon and Funk transform inversions, respectively. Then, 
\begin{equation}\label{eq:midinversionformula}
    Rf(\omega,s+p\cdot \omega)=F^{-1}_{v\rightarrow \omega} \Big\{\int_s^\infty Cf(v,p,r)g(s,r)\, dr\Big\},
\end{equation}
and thus 
\begin{equation}\label{eq:inversionformula}
    f=R^{-1}\bigg\{F^{-1}_{v\rightarrow \omega} \Big\{\int_s^\infty Cf(v,p,r)g(s,r)\, dr\big\}\bigg\}.
\end{equation}
\end{theorem}
\begin{proof}
    The proof follows directly from Proposition \ref{Prop:Cf=RFf}.
\end{proof}

\begin{remark}\label{remark1}
    To reconstruct \(f\) from the Radon data obtained in \eqref{eq:midinversionformula}, complete Radon data are needed. Since \(\operatorname{supp} f \subset B(0,a)\), it is enough for the values \(s+p\cdot \omega\) to cover the interval \([-a,a]\). Because \(p\in \mathbb{S}_a^{n-1}\), this condition is already satisfied when \(s=0\).
\end{remark}

\section{Numerical Experiments}\label{numerical-results}

In this section, we present the results of the numerical implementation of the analytical inversion formula \eqref{eq:inversionformula} in three dimensions using a ball phantom. 
\subsection{Simulation of Forward Measurements}
The imaging domain \(B(0,a)\) is taken to be the unit ball, corresponding to \(a=1\) in Definition \ref{cylindrical Radon transform}. The phantom is chosen as \(f=\mathcal{X}_{B(c,t)}\), the indicator function of a ball centered at
\(c=(-0.2,0.2,0.3)\)
with radius
\(t=\frac12\) (see Figure \ref{3D phantom and imaging domain}). Since the imaging domain is contained in the unit ball, cylinders of radius \(r>2\) do not intersect the domain, and hence the forward data vanish for \(r>2\).

\begin{figure}[htbp]
\centering
    \begin{subfigure}{0.38\textwidth}
        \centering
        \includegraphics[width=\linewidth]{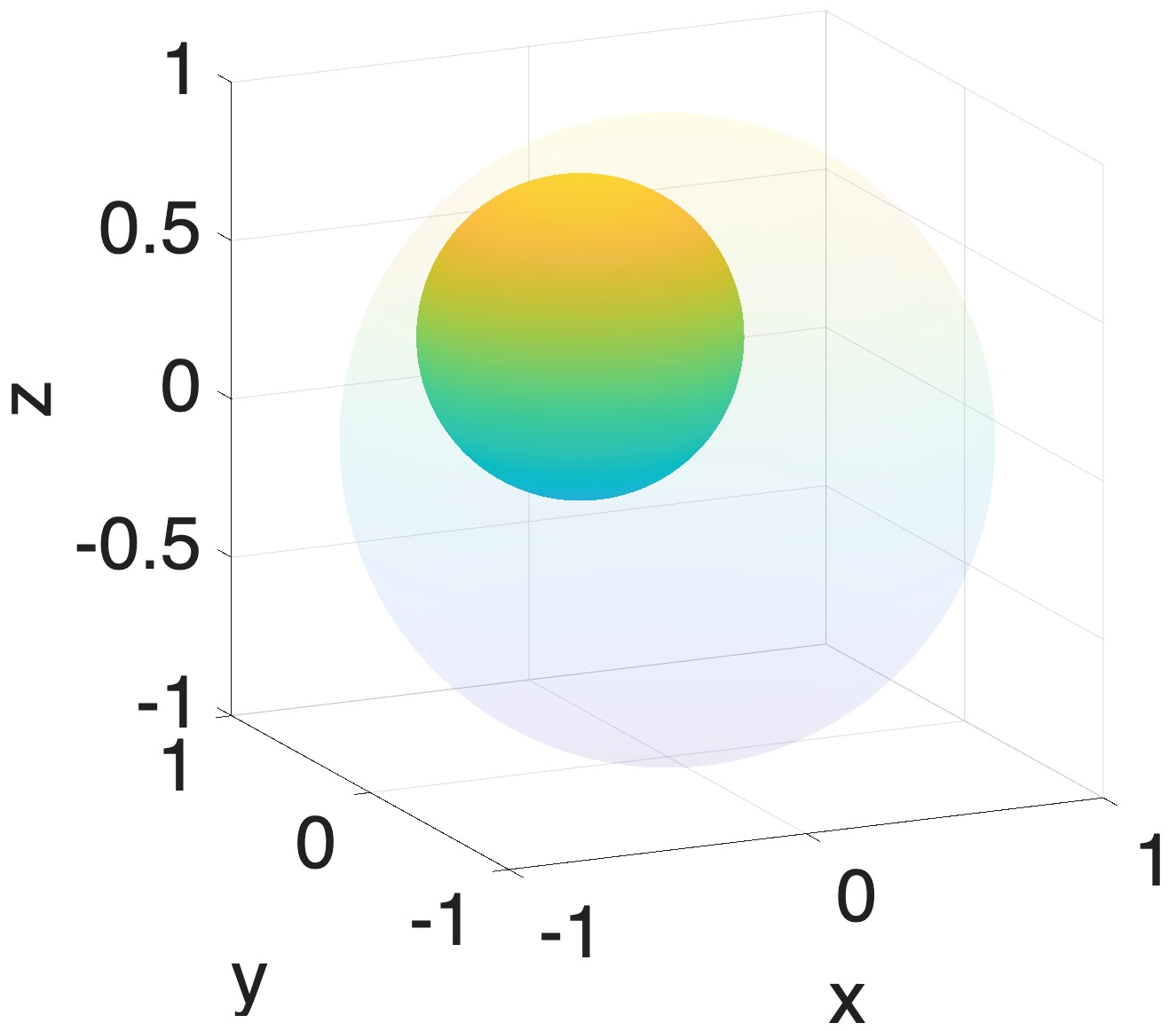}
        \caption{}
        \label{3D phantom and imaging domain}
    \end{subfigure}
    \begin{subfigure}{0.38\textwidth}
        \centering
        \includegraphics[width=\linewidth]{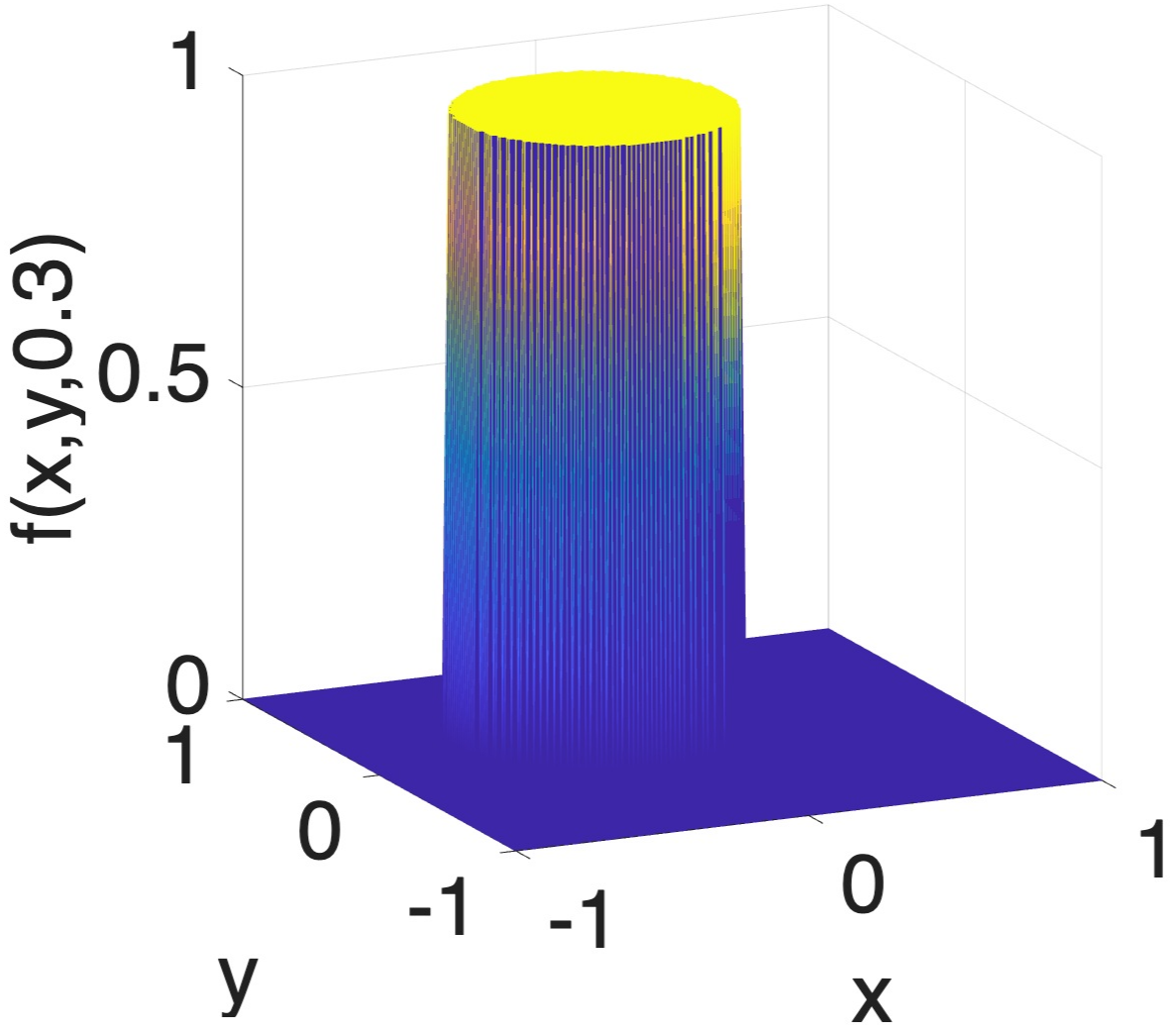}
        \caption{}
        \label{xy-phantom}
    \end{subfigure}
    \begin{subfigure}{0.38\textwidth}
        \centering
        \includegraphics[width=\linewidth]{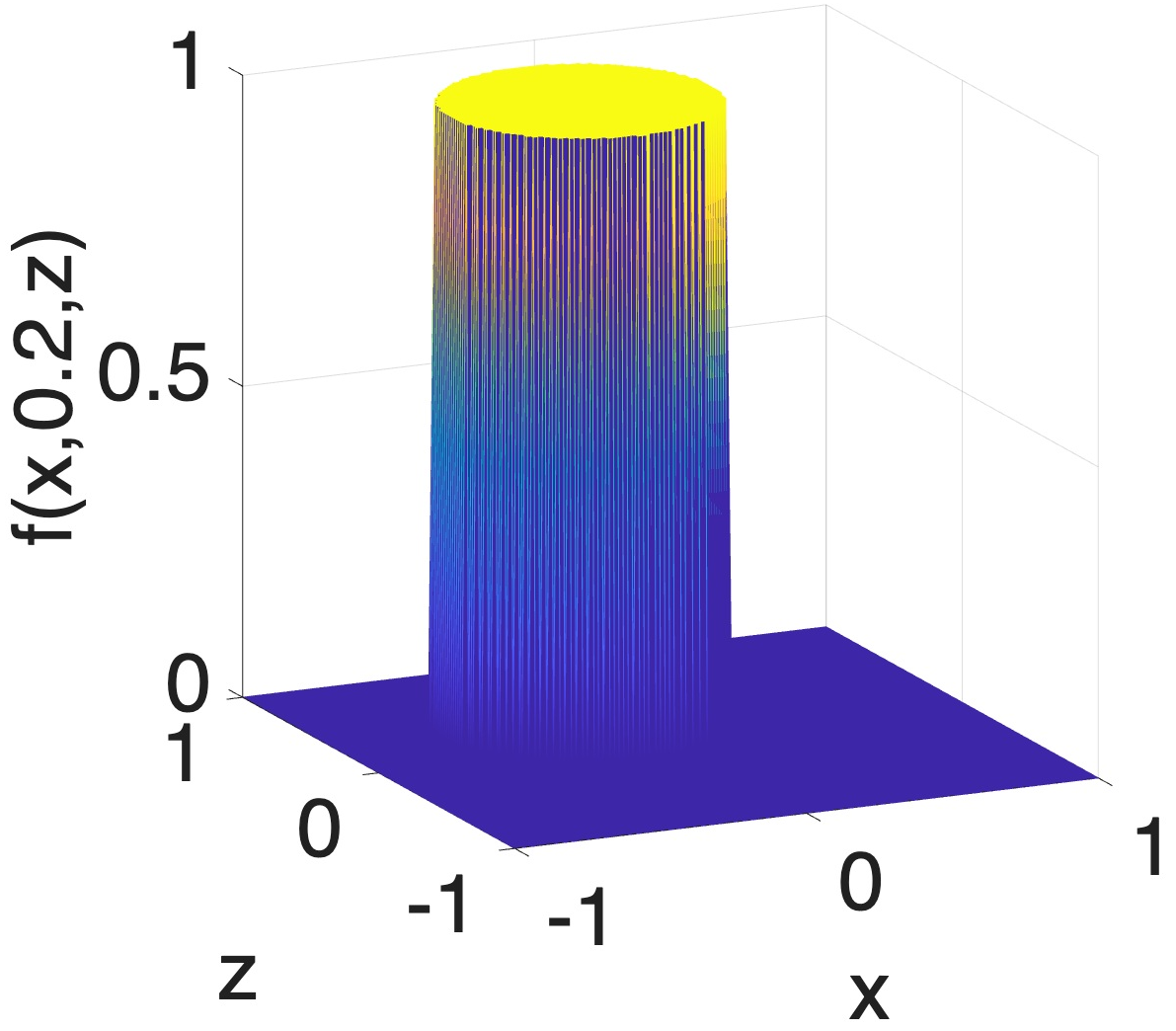}
        \caption{}
    \end{subfigure}
    \begin{subfigure}{0.38\textwidth}
        \centering
        \includegraphics[width=\linewidth]{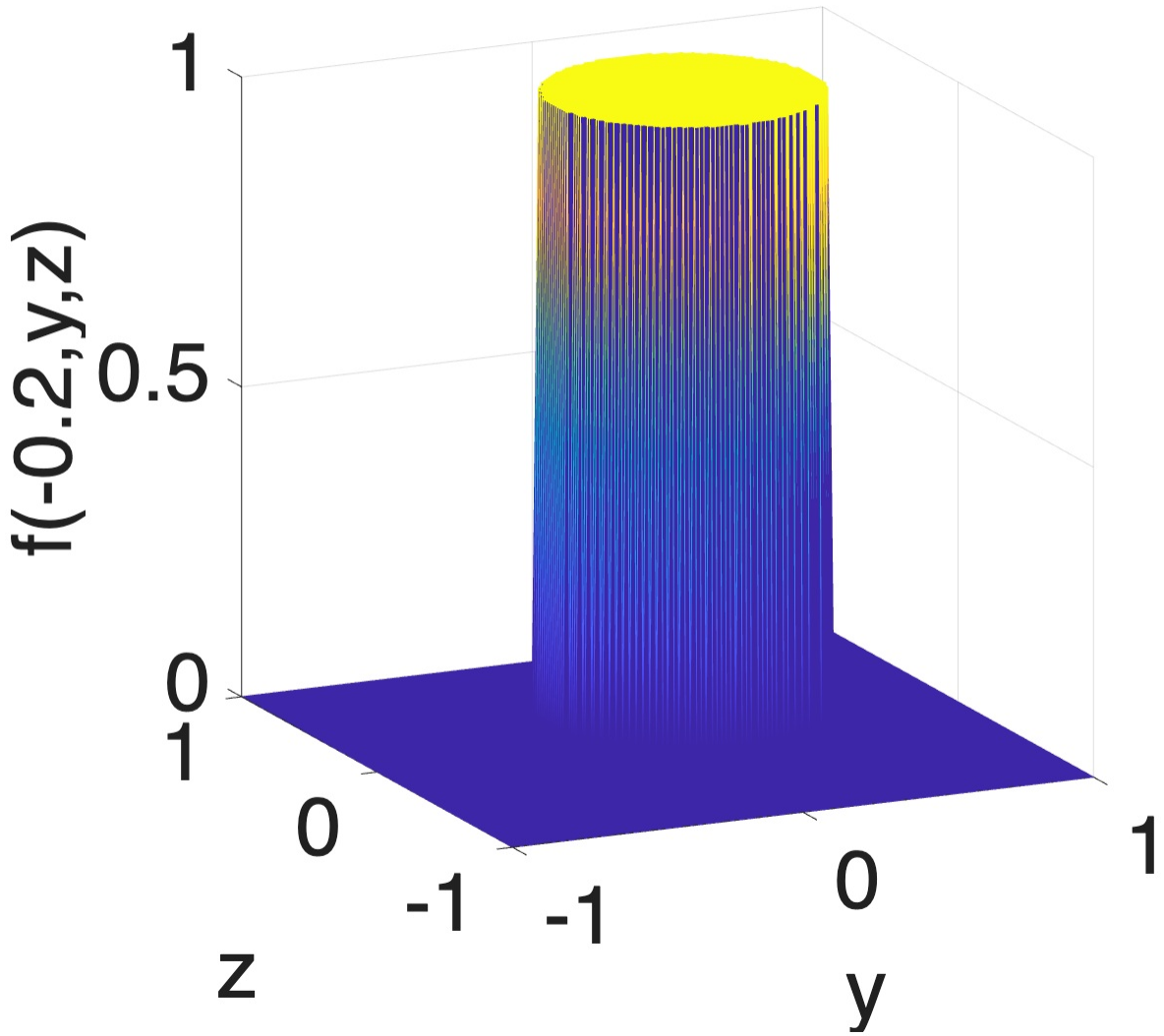}
        \caption{}
        \label{3D-yz-phantom}
    \end{subfigure}
    \caption{(a) Imaging domain \(B(0,a)\) and ball phantom \(f=\mathcal{X}_{B(c,t)}\) where \(c=(-0.2,0.2,0.3)\) and \(t=0.5\). (b) Surface plot of \(f(x,y,0.3)\). (c) Surface plot of \(f(x,0.2,z)\). (d) Surface plot of \(f(-0.2,y,z)\).}
    \end{figure}

The cylindrical Radon transform variables \((v,p,r)\in \mathbb{S}^2\times \mathbb{S}^2\times (0,2]\) were sampled as follows. The points \(p\) and directions \(v\) are sampled using Fibonacci lattices, yielding quasi-uniform distributions on \(\mathbb{S}^2\). In our implementation, we used \(13{,}000\) directions \(v\in\mathbb{S}^2\) and \(500\) points \(p\in\mathbb{S}^2\). We uniformly sampled \(500\) radius values \(r\in(0,2]\).

We computed the forward data by numerically computing the elliptic integral \eqref{eq:ball CRT}. 
Figure \ref{Transform data} shows the plot of \(Cf(v,p,r)\) as a function of the radius \(r\) for fixed \(p\) and \(v\). Figure \ref{Noisy transform data} shows the same profile for the noisy forward data where white Gaussian noise with signal-to-noise ratio (SNR) of \(20\) dB was added to the 3D array.
\begin{figure}[htbp]
\centering
\begin{subfigure}{0.42\textwidth}
        \centering
        \includegraphics[width=\linewidth]{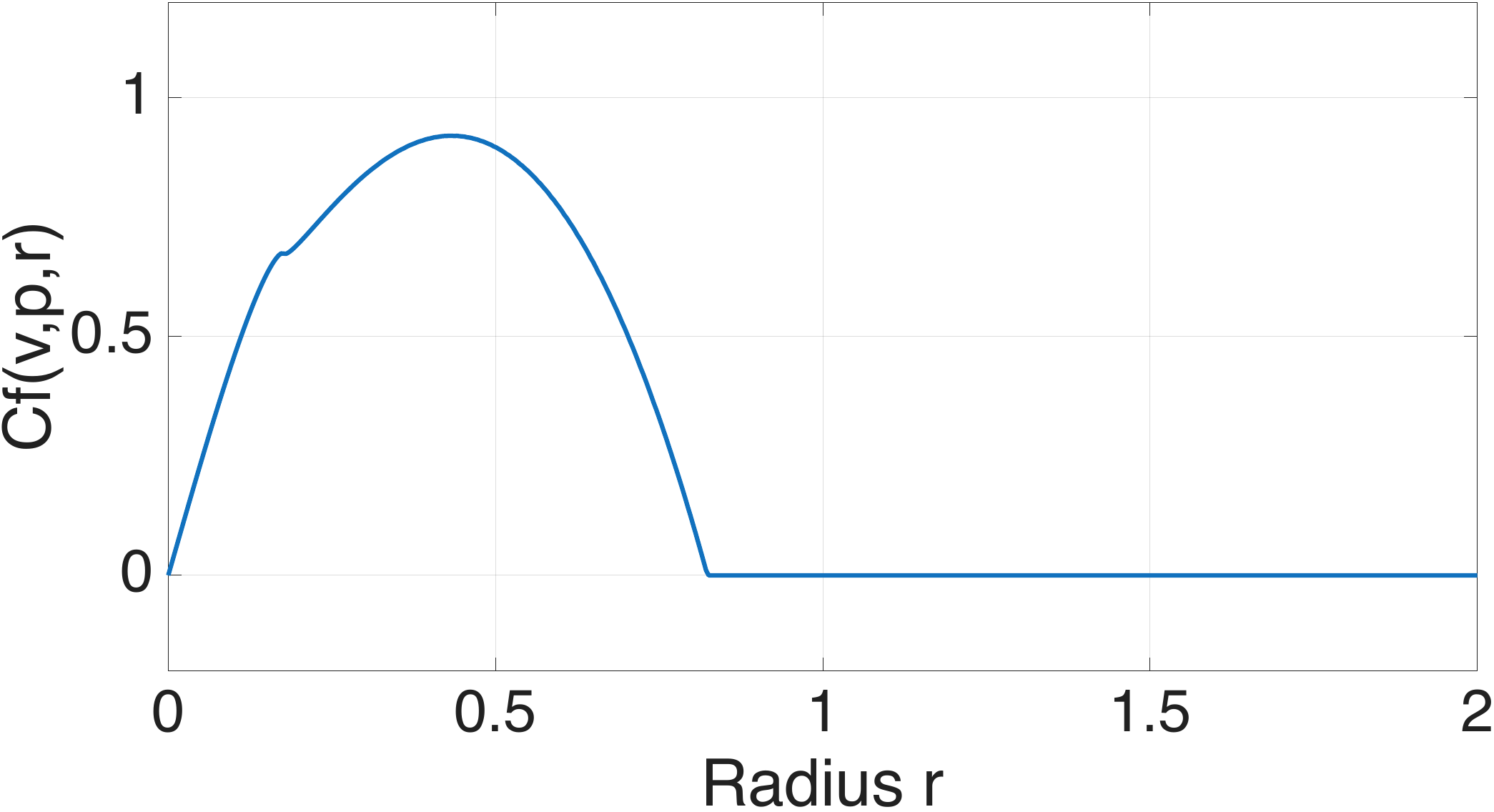}
        \caption{}
        \label{Transform data}
    \end{subfigure}
    \begin{subfigure}{0.42\textwidth}
        \centering
        \includegraphics[width=\linewidth]{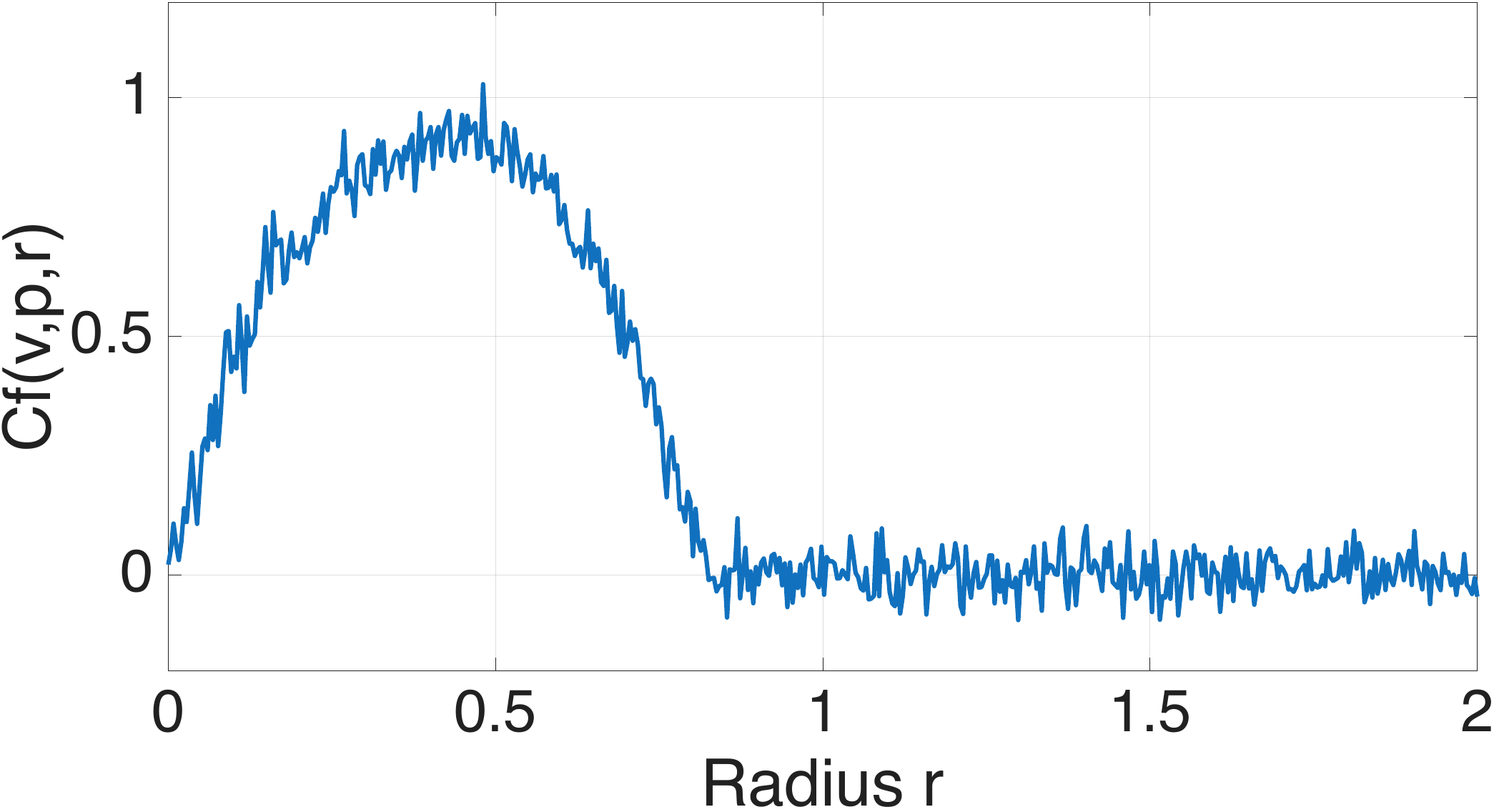}
        \caption{}
        \label{Noisy transform data}
    \end{subfigure}
    \caption{The cylindrical Radon transform data \(Cf(v,p,r)\) for varying \(r\) when \(v\approx(0.01,0,1)\) and \(p\approx(0.06,0, 1)\) (a) without added noise, and (b) with Gaussian noise added (SNR 20 dB).} 
        \label{fig:CRT-visual}
    \end{figure}
    
    The behavior of the forward data in Figure \ref{Transform data} around \(r=0.25\) is due to one side of the cylinder becoming tangent to the boundary of the ball phantom at that radius. We test our inversion algorithm using both the noiseless and noisy forward data. 
 \subsection{Implementation of the Inversion Algorithm} We now describe the numerical implementation of the inversion algorithm. The integral relation in three dimensions reads as \begin{equation}\label{3Dintrelation}\int_{s}^{2} Cf(v,p,r)g(s,r)dr = F(R_s T_{p} f)(v),\quad g(s,r)=2(r^2-s^2)^{-1/2}.\end{equation} We first numerically approximate the weighted integral on the left hand side of \eqref{3Dintrelation}. Although \eqref{3Dintrelation} holds for all \(0\leq s\leq 2\), in view of Remark \ref{remark1}, it suffices to take \(s=0\). This simplifies the weight function to \(g(0,r)=\frac{2}{r}\)
reducing the computational complexity of the weighted integration step while still providing complete data needed for the Radon transform inversion.  We write \begin{equation}\label{eq:G}
    G(v,p):=\int_0^2 Cf(v,p,r)\frac{2}{r}\, dr=F(RT_{p}f)(v).
\end{equation} 
For each \(p\), \(G(v,p)\) is even with respect to 
\(v\) since the cylindrical Radon transform is even with respect to \(v\) (see Proposition \ref{prop:properties}). The application of the Funk transform inversion formula \eqref{ndimFTI} in three dimensions with respect to \(v\) yields: \begin{equation}\label{eq:R=F-1}
    Rf(\omega,\omega \cdot p)=RT_{p}f(\omega)= \frac{1}{2\pi}\Big[\frac{d}{du}\int_0^u \hat{F}G_{\cos^{-1}t} (\omega,p)(u^2-t^2)^\frac{-1}{2}\,t\, dt\Big]_{u=1},
\end{equation} where \(\omega\in\mathbb{S}^2\) is sampled using a Fibonacci lattice for \(9,000\) directions. Since we work with ball phantoms, we have a closed-form expression  for the Radon transform data, which we use as an intermediate check point. 

The Radon transform of a three-dimensional ball can be computed analytically:
\begin{equation}\label{eq:Rfball3D}
    R\mathcal{X}_{B(c,t)}(\omega,s)=
\bigg\{
\begin{array}{ll}
\pi(t^2-(s-\omega\cdot c)^2),\quad |s-\omega\cdot c|\leq t, \\
0, \quad \text{otherwise},
\end{array}
\end{equation}
see the Appendix~\ref{app:proof of 3D ball} for the proof.

\begin{figure}[htbp]
\centering
    \begin{subfigure}{0.32\textwidth}
        \centering
        \includegraphics[width=\linewidth]{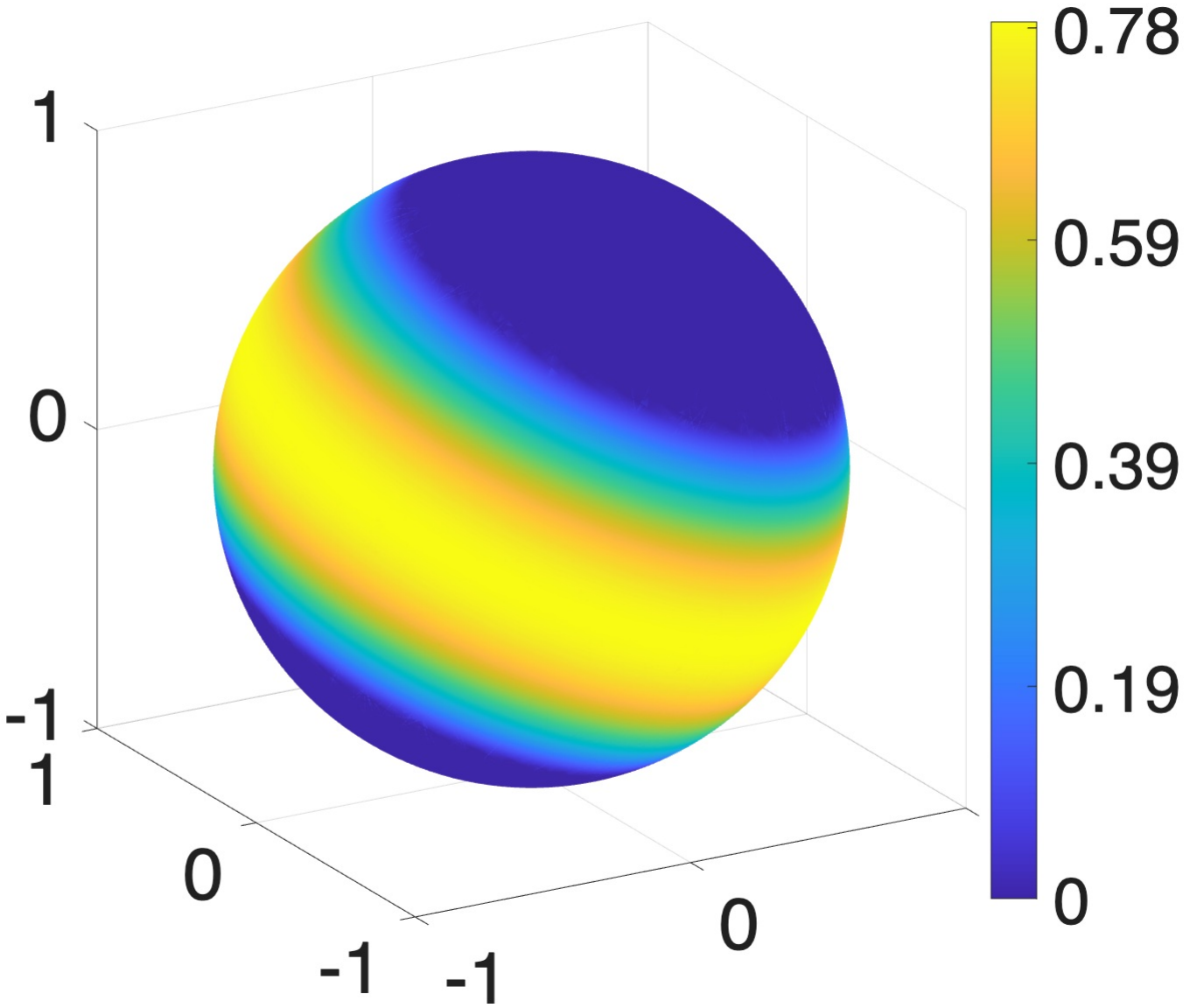}
        \caption{}
        \label{True Rf}
    \end{subfigure}
    \begin{subfigure}{0.32\textwidth}
        \centering
        \includegraphics[width=\linewidth]{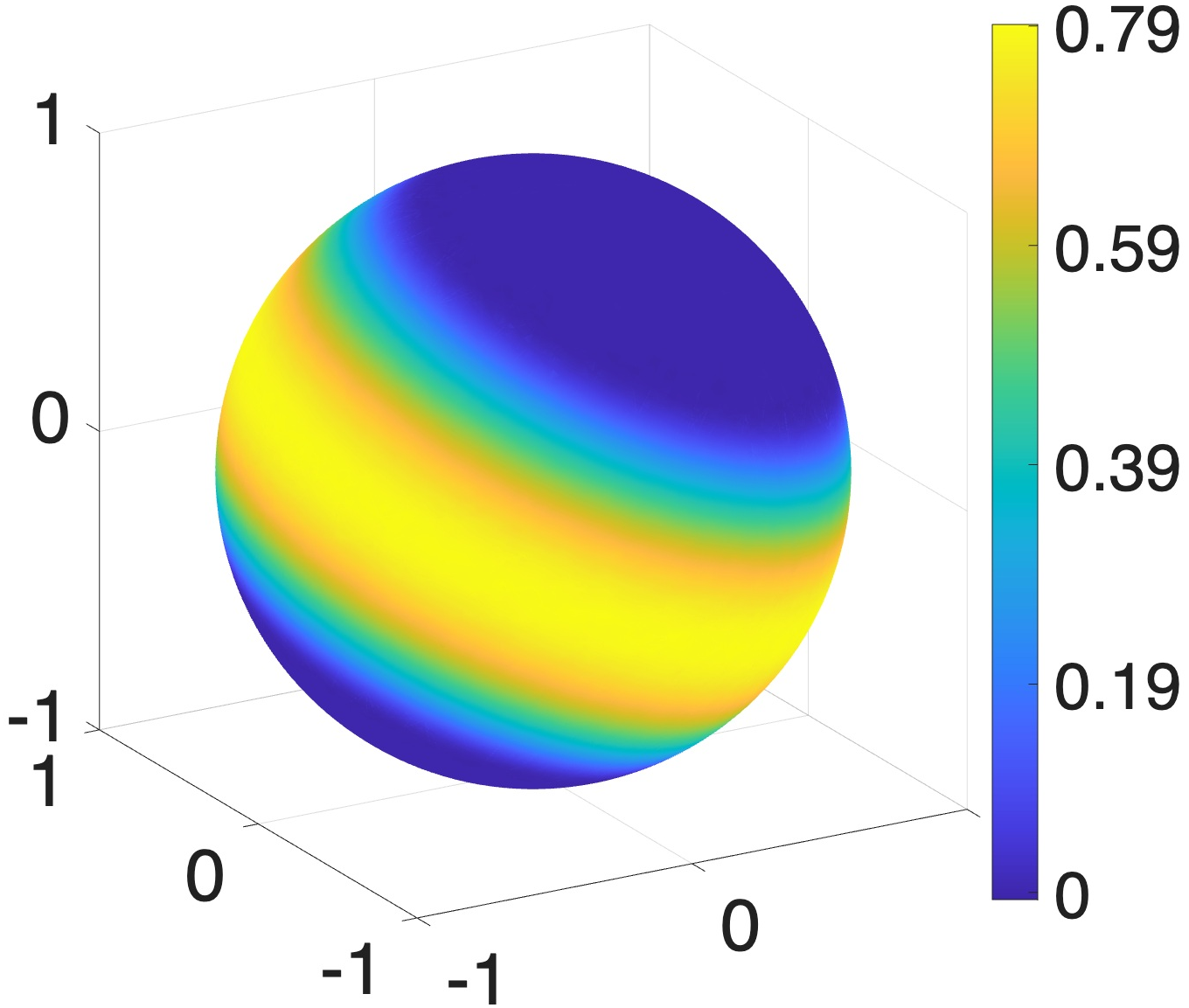}
        \caption{}
        \label{clean Rf}
    \end{subfigure}
    \begin{subfigure}{0.32\textwidth}
        \centering
        \includegraphics[width=\linewidth]{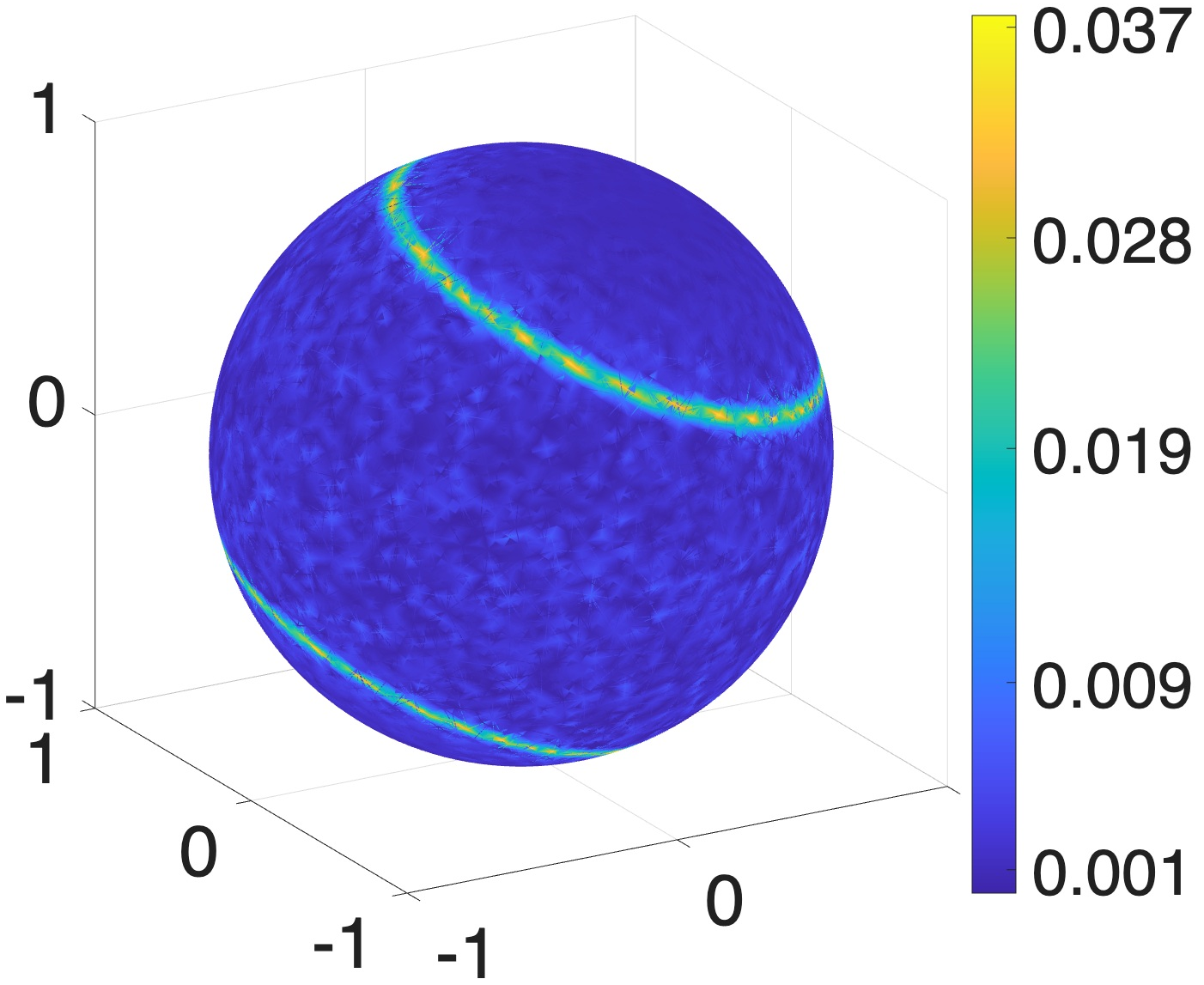}
        \caption{}
        \label{abs error}
    \end{subfigure}

    \caption{(a) Exact Radon data for \(f=\mathcal{X}_{B(c,t)}\) where \(c=(-0.2,0.2,0.3)\) and \(t=0.5\). (b) Radon data recovered via \eqref{eq:R=F-1} from noiseless cylindrical Radon transform data. (c) Absolute difference between exact Radon data and recovered Radon data for \(p\approx(0.06,0,1)\),\,\,\#\(\omega=9000\).}
    \label{fig:radon data}
\end{figure}

Figure \ref{fig:radon data} shows the plots of the exact Radon data, its reconstruction from cylindrical Radon transform data via \eqref{eq:R=F-1}, and their absolute difference when \(p\approx(0.06,0,1)\). We see that the error accumulates around the singularities of the ball phantom. The maximum relative \(L^2\)-error (computed over \(\omega\) for each fixed \(p\)) across all \(p\) is approximately \(0.0407\). Before executing the Radon transform inversion step, we observe that the data 
\(Rf(\omega,\omega\cdot p)\) is not uniformly spaced in the second variable. Thus, we resample the Radon data to obtain a uniformly spaced linear variable \(s\in[-a,a] \) before implementing the filtered backprojection algorithm. The resampling is done by projecting samples onto a uniform \(s\)-grid using binning with local averaging, followed by interpolation. This resulted in a relative \(L^2\)-error of 0.0241, which is computed numerically via the formula 
\[ L_{rel}^2(Rf_{resampled},Rf_{exact})= \sqrt{\frac{\int_\mathbb{R}\int_{\mathbb{S}^2}|Rf_{resampled}-Rf_{exact}|^2\, d\omega\, ds}{\int_\mathbb{R}\int_{\mathbb{S}^2}|Rf_{exact}|^2\, d\omega\, ds}}.\] 

Figure \ref{fig:resampling} visualizes the exact, recovered, and resampled Radon data for a fixed direction \(\omega\). We see that the error is larger around the edges of the support of the Radon data which again corresponds to the singularity of the ball phantom. 
\begin{figure}[htbp]
\centering\includegraphics[width=.6\textwidth]{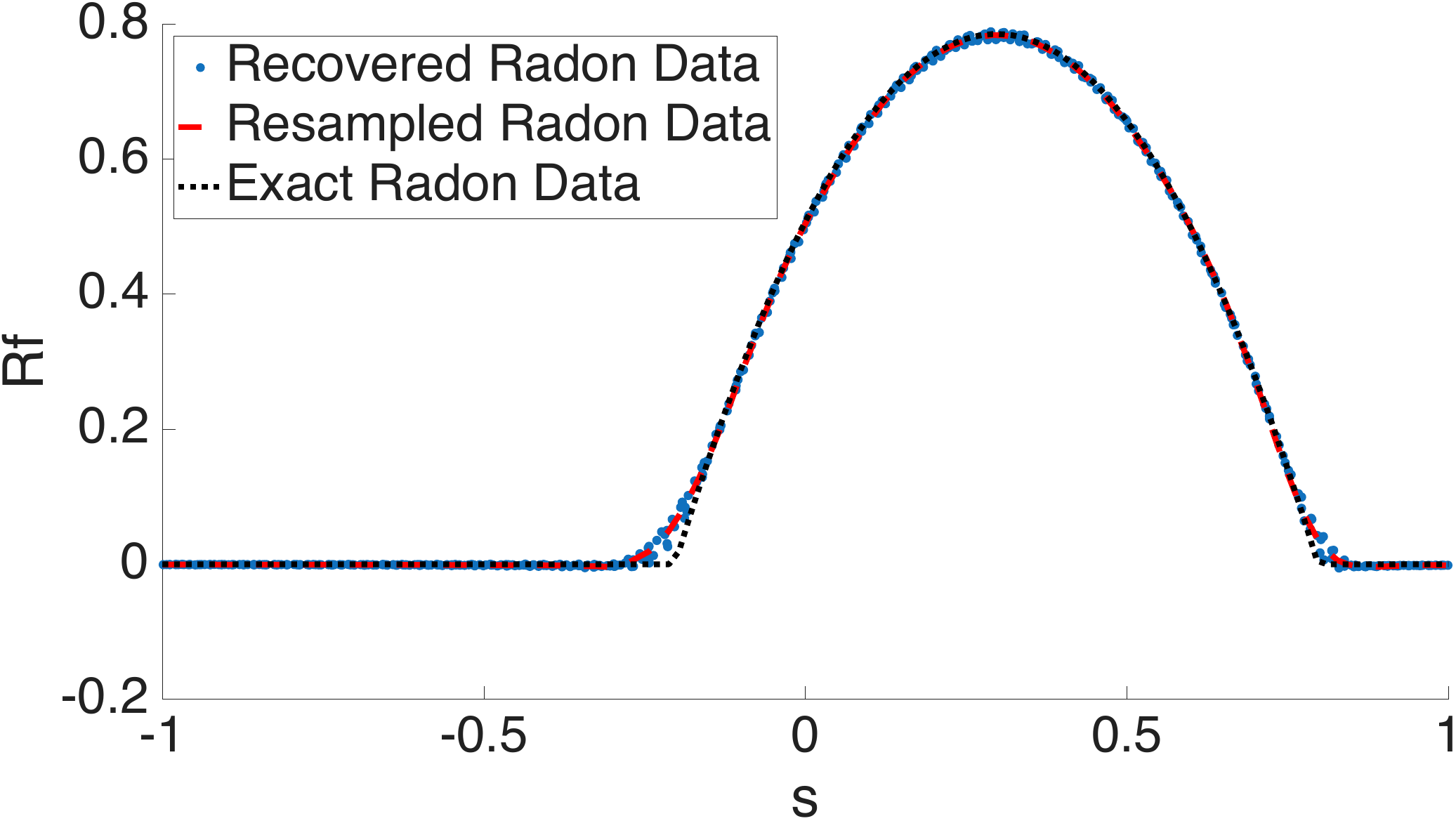} 
\caption{\centering Recovered Radon data and its resampling when \(\omega\approx(0.015,	0,1).\)} 
\label{fig:resampling}
\end{figure}

Finally, to the resampled Radon data, we apply the inversion formula for the Radon transform in three dimensions \eqref{ndimradoninversion}: \begin{equation*}
    f(x)=-\frac{1}{8\pi^2}\int_{\mathbb{S}^2} \frac{\partial^2}{\partial s^2} Rf(\omega,s)\Big|_{s=x\cdot \omega}\, d\omega,
\end{equation*} 
using the filtered backprojection algorithm \cite{natterer_mathematical_2001}. In the numerical implementation, the resampled Radon data were evaluated on a uniform \(s\)-grid on \([-a,a]\). The second derivative in \(s\) was approximated through the standard ramp-filtering step in the filtered backprojection algorithm, followed by interpolation at \(s=x\cdot\omega\) and numerical integration over the sampled directions \(\omega\in\mathbb{S}^2\).
\subsection{Reconstruction from Noiseless Data}
We now compare the reconstructions with the exact ball phantom. Figure \ref{fig:clean surface plots} shows the surface plots of two dimensional cross sections through the center of the phantom. 
\begin{figure}[ht]
\centering
    \begin{subfigure}{0.32\textwidth}
        \centering
        \includegraphics[width=\linewidth]{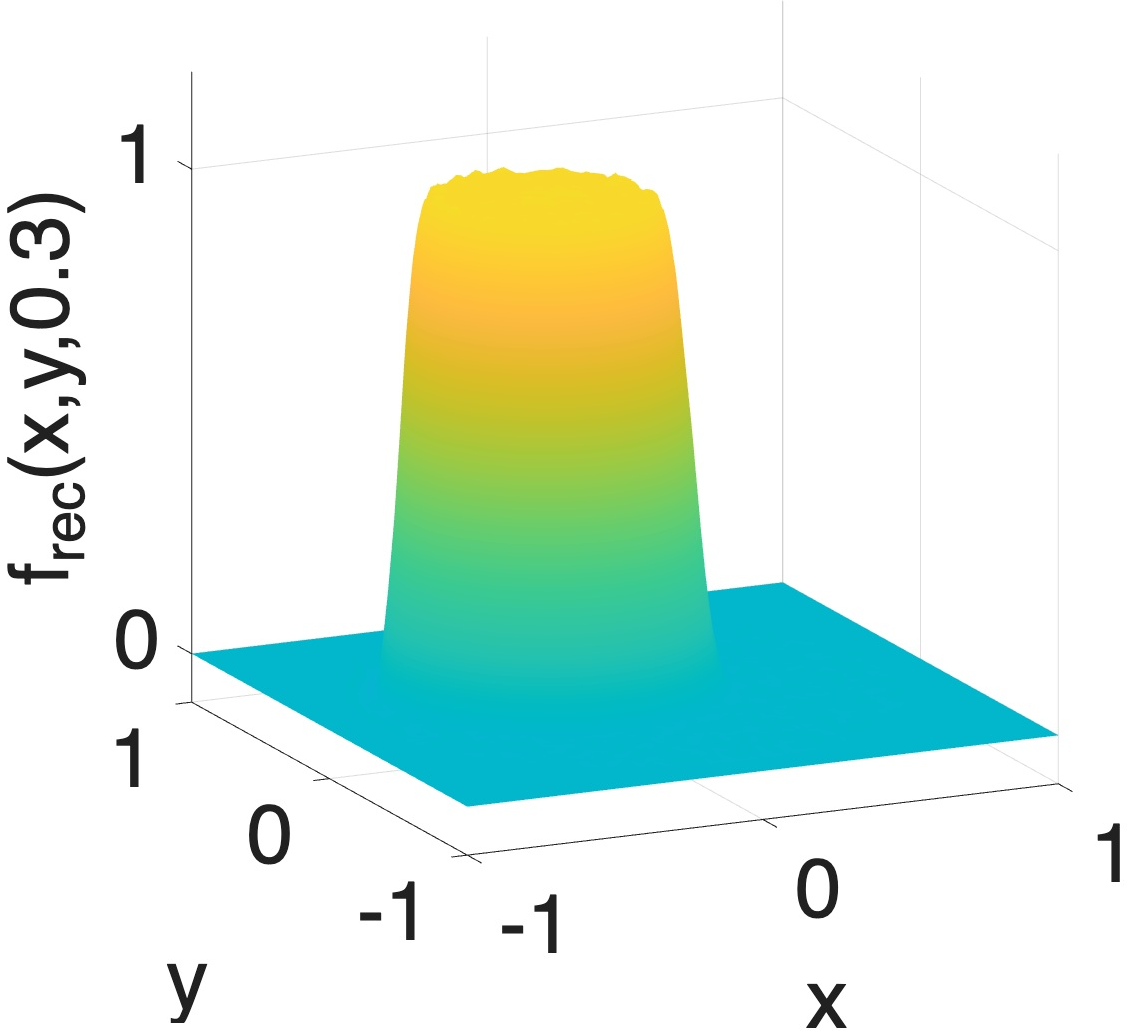}
        \caption{}
    \end{subfigure}
    \begin{subfigure}{0.32\textwidth}
        \centering
        \includegraphics[width=\linewidth]{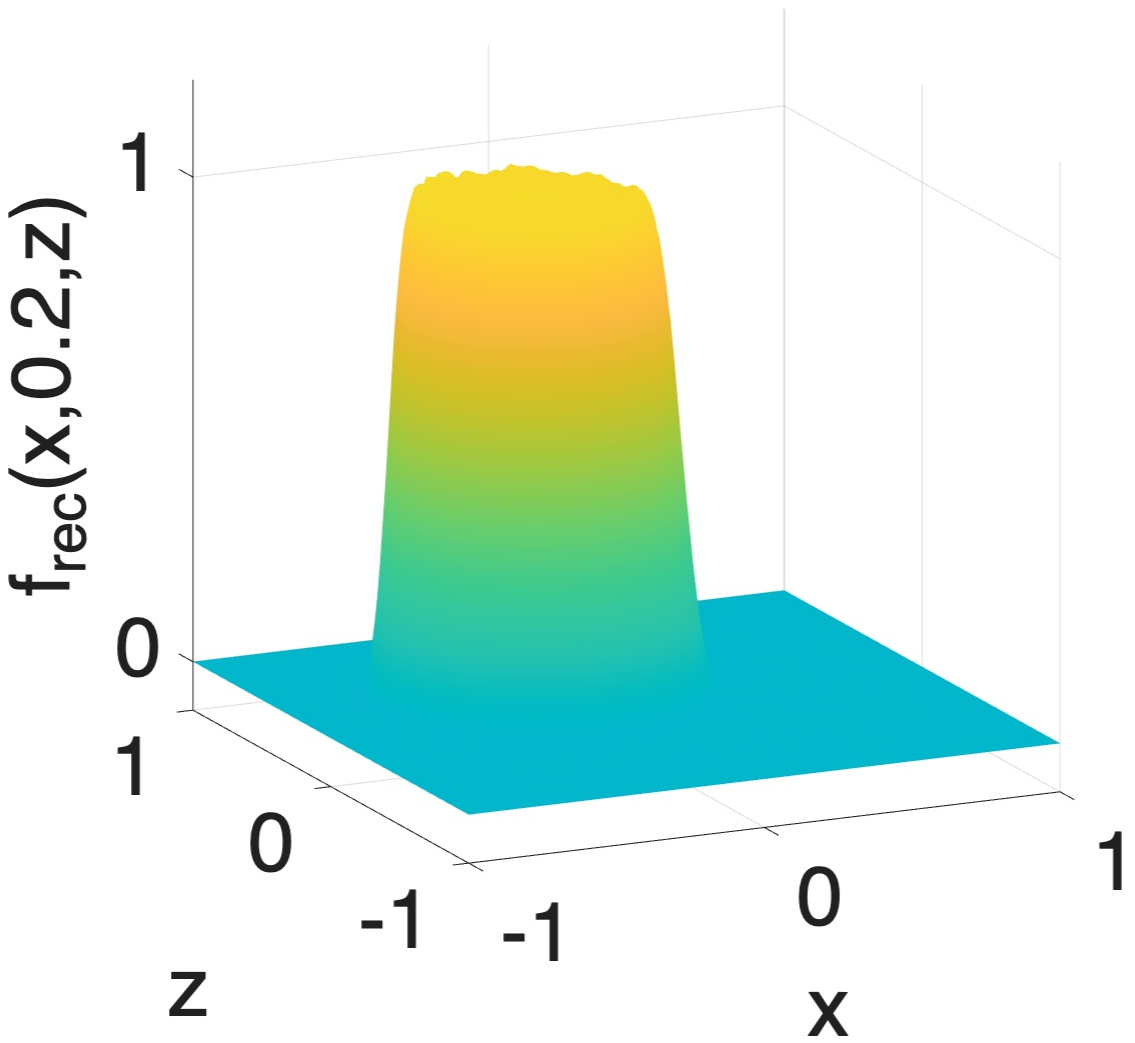}
        \caption{}
    \end{subfigure}
    \begin{subfigure}{0.32\textwidth}
        \centering
        \includegraphics[width=\linewidth]{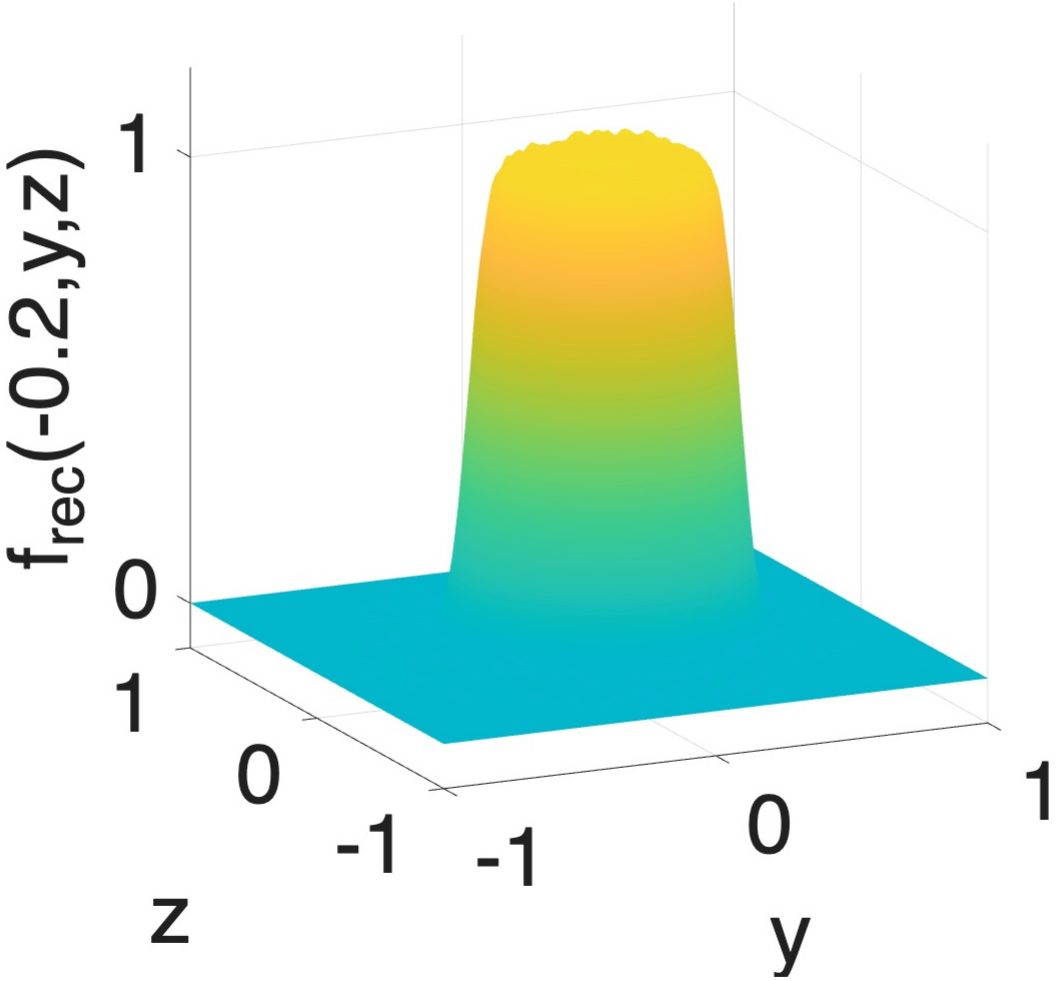}
        \caption{}
    \end{subfigure}

    \caption{The surface plots of the reconstruction from noiseless cylindrical Radon transform data using cross sections through the center of the phantom. (a) \(f_{rec}(x,y,0.3)\), (b) \(f_{rec}(x,0.2,z)\), (c) \(f_{rec}(-0.2,y,z)\).}
    \label{fig:clean surface plots}
\end{figure}
The results show that the inversion algorithm is effective in reconstructing the overall geometry and location of the phantom, although smoothing effects are visible near the boundary of the slices.

\begin{figure}[ht]
    \centering
    \begin{subfigure}{0.3\textwidth}
        \centering
        \includegraphics[width=\linewidth]{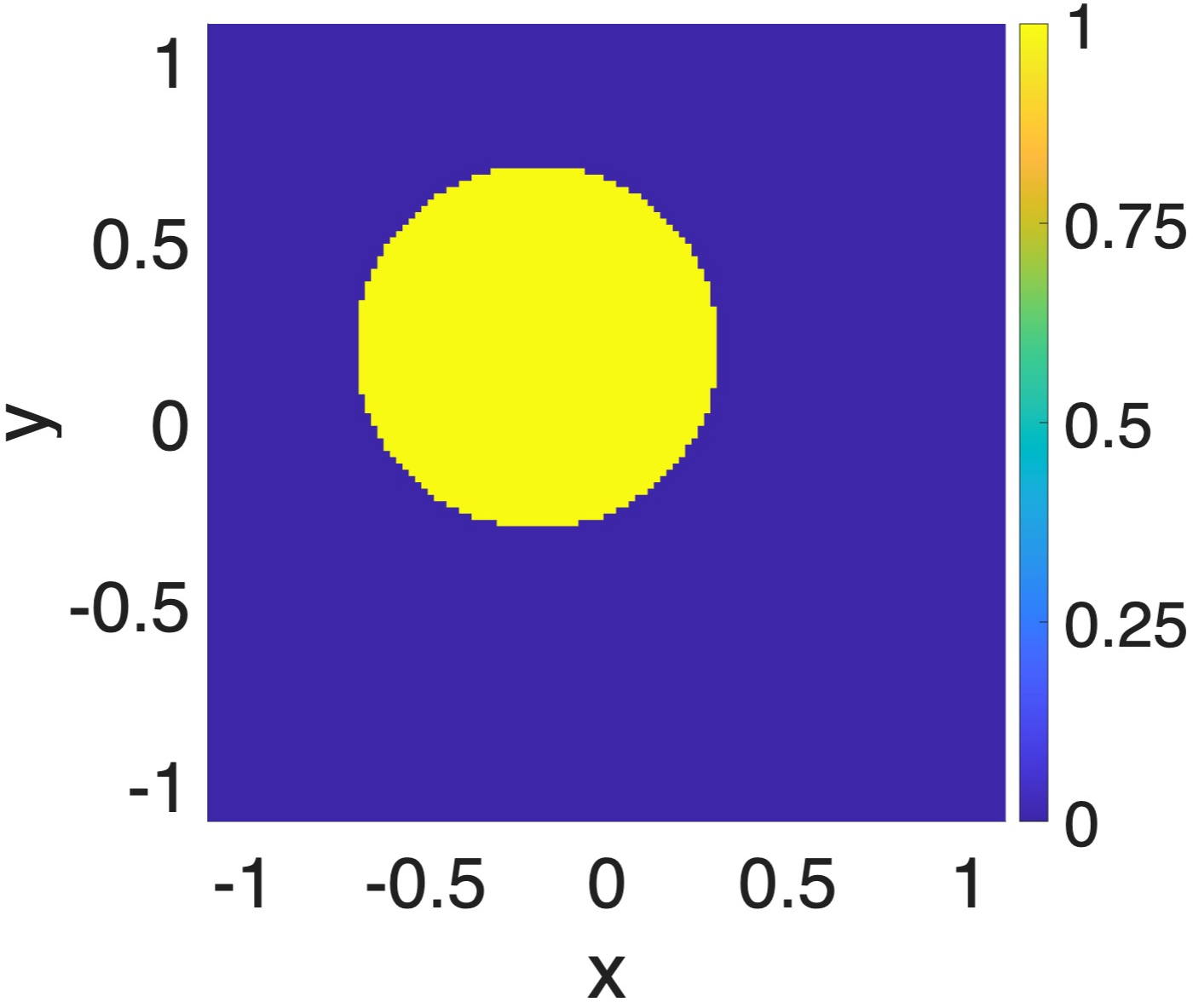}
        \caption{}
    \end{subfigure}
    \begin{subfigure}{0.3\textwidth}
        \centering
        \includegraphics[width=\linewidth]{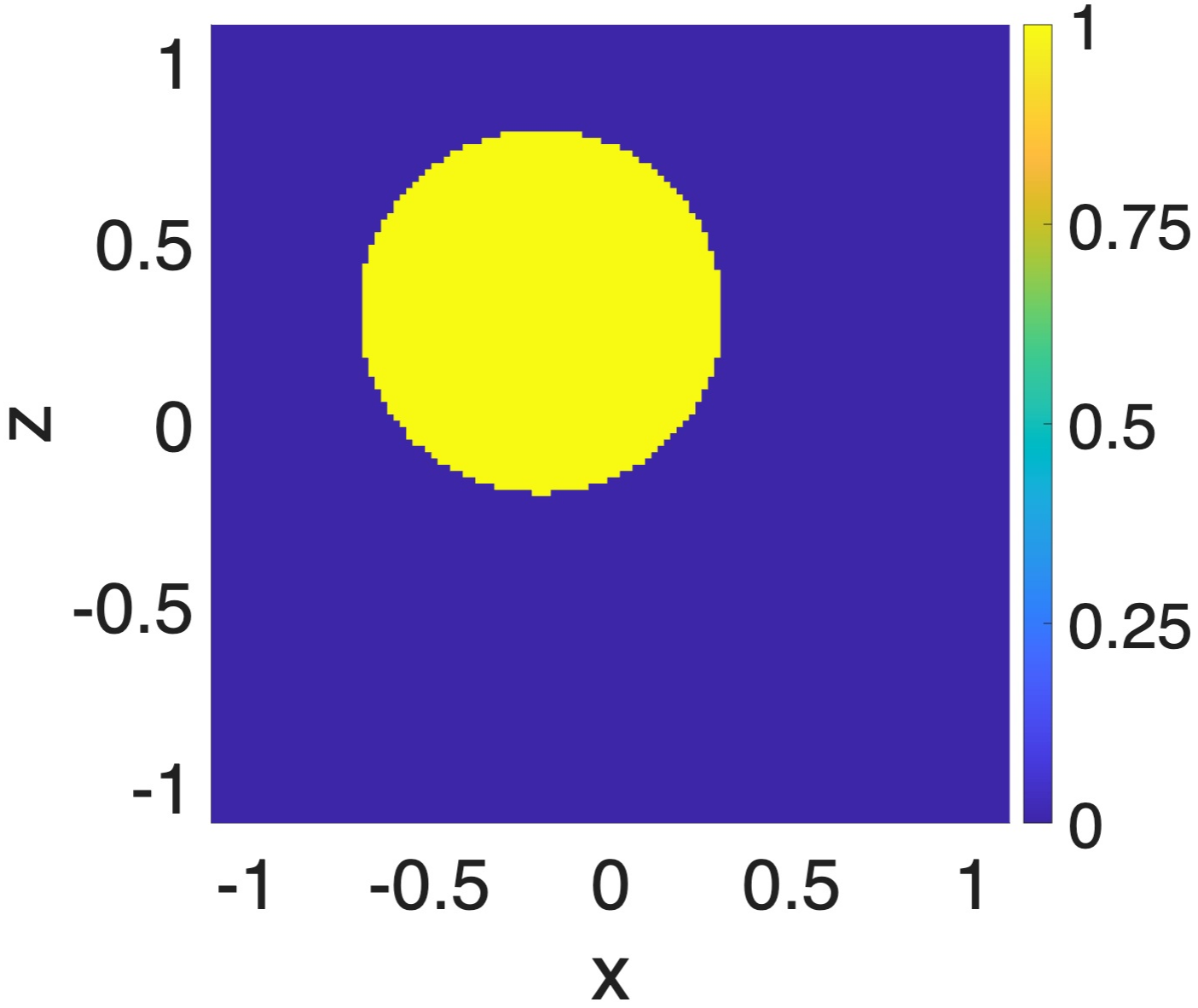}
        \caption{}
    \end{subfigure}
    \begin{subfigure}{0.3\textwidth}
        \centering
        \includegraphics[width=\linewidth]{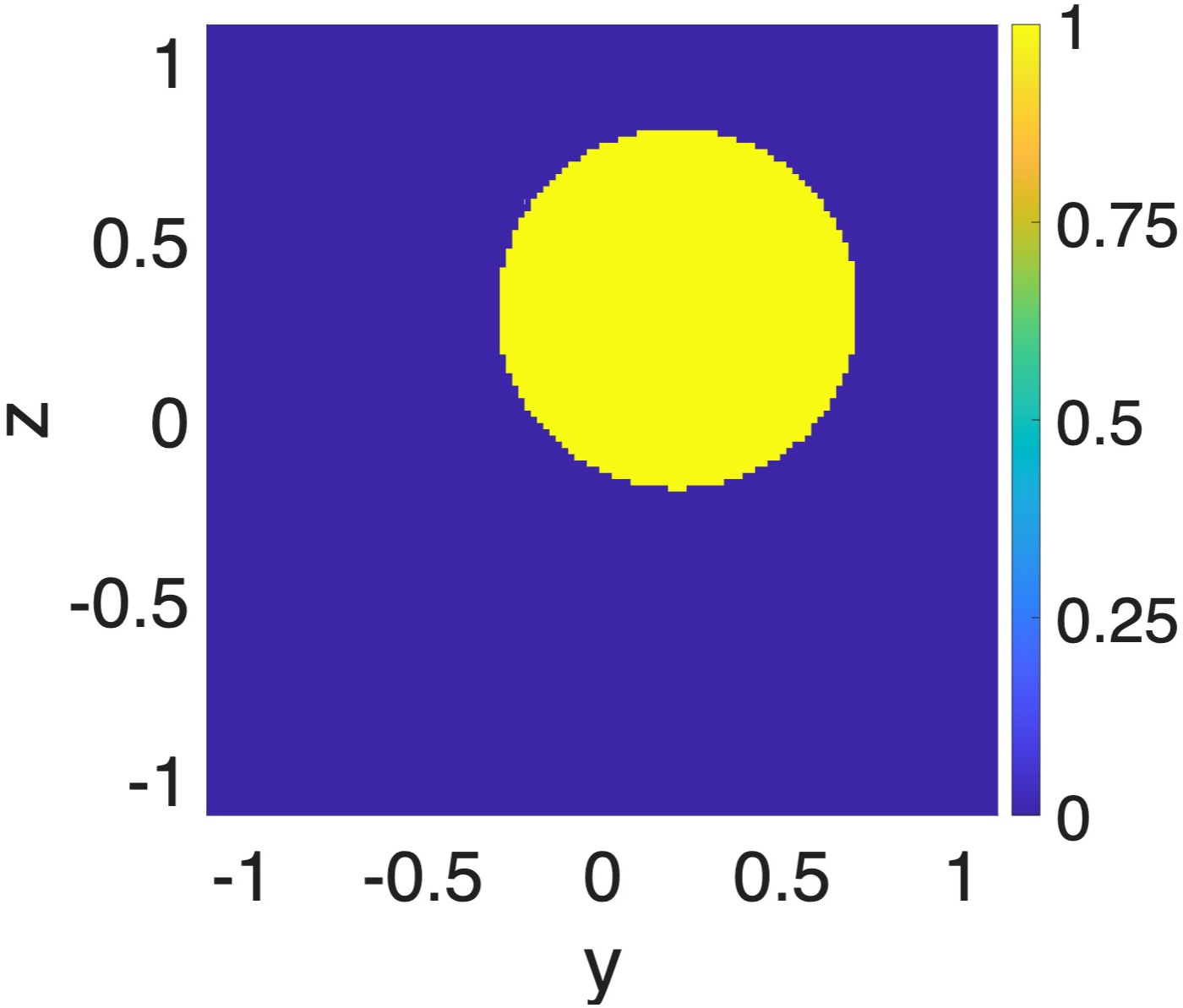}
        \caption{}
    \end{subfigure}
    \begin{subfigure}{0.3\textwidth}
        \centering
        \includegraphics[width=\linewidth]{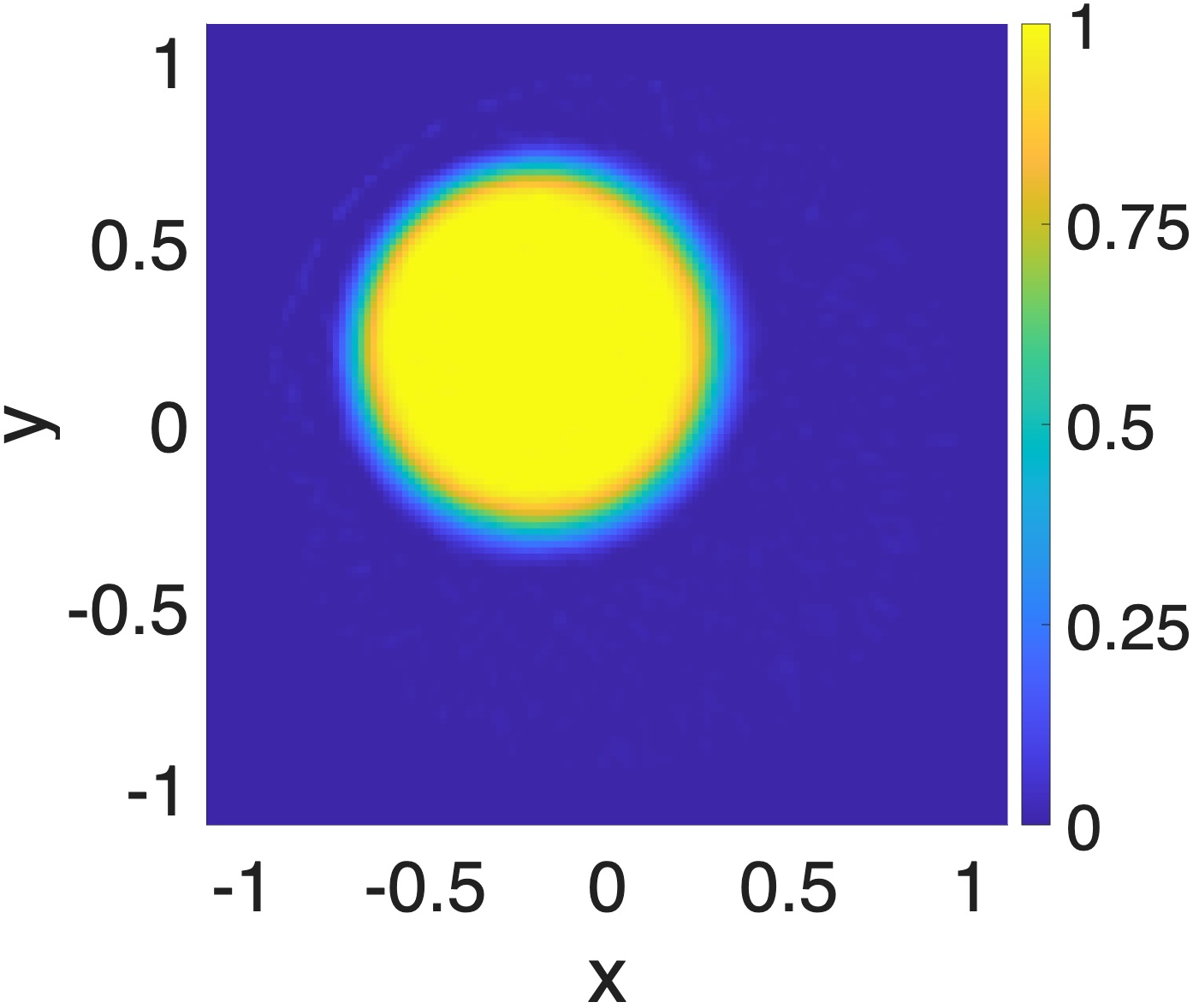}
        \caption{}
        \label{xy-phantom}
    \end{subfigure}
    \begin{subfigure}{0.3\textwidth}
        \centering
        \includegraphics[width=\linewidth]{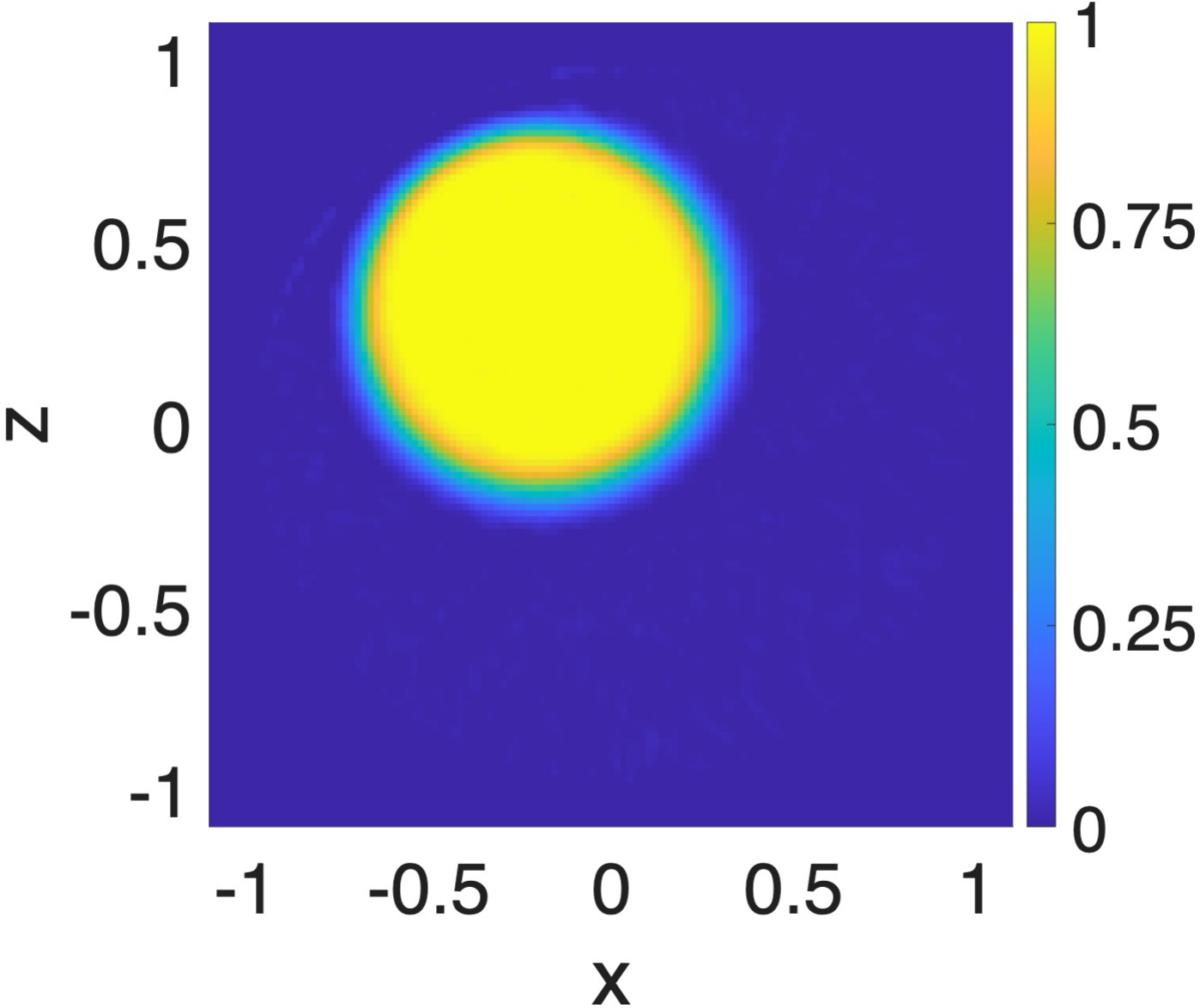}
        \caption{}
        \label{xy-phantom}
    \end{subfigure}
    \begin{subfigure}{0.3\textwidth}
        \centering
        \includegraphics[width=\linewidth]{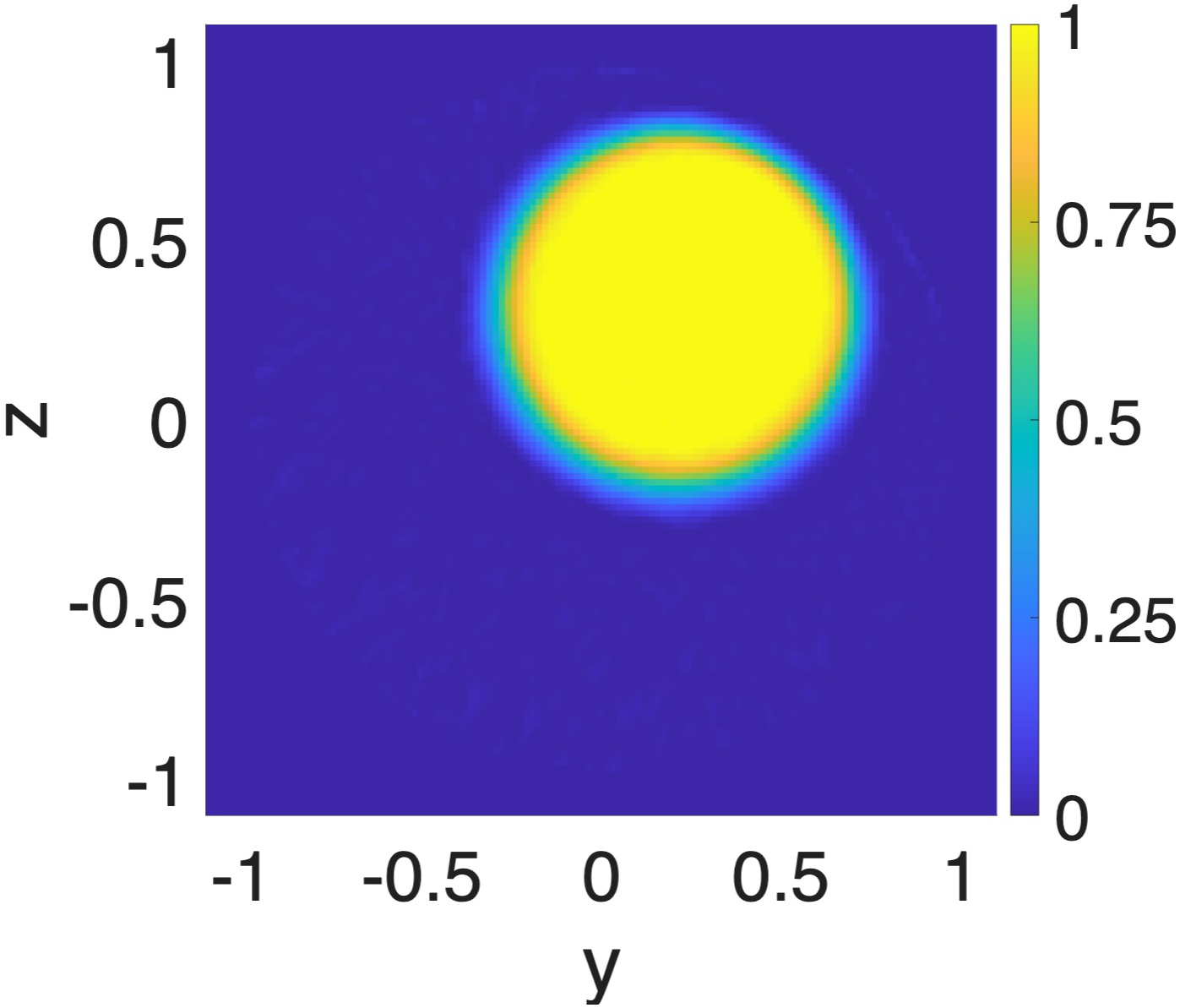}
        \caption{}
        \label{xy-phantom}
    \end{subfigure}
    \begin{subfigure}{0.3\textwidth}
        \centering
        \includegraphics[width=\linewidth]{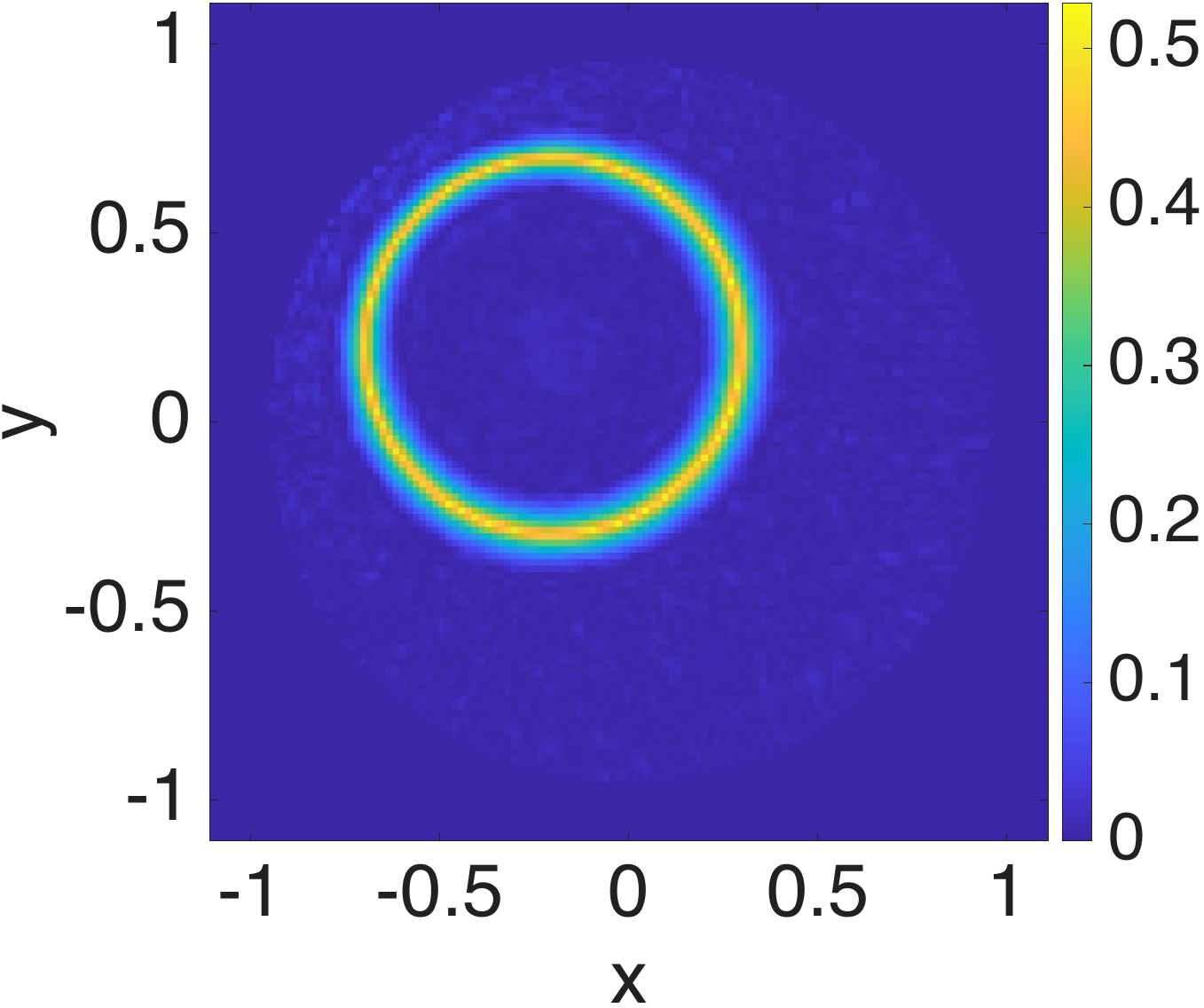}
        \caption{}
        \label{xy-phantom}
    \end{subfigure}
    \begin{subfigure}{0.3\textwidth}
        \centering
        \includegraphics[width=\linewidth]{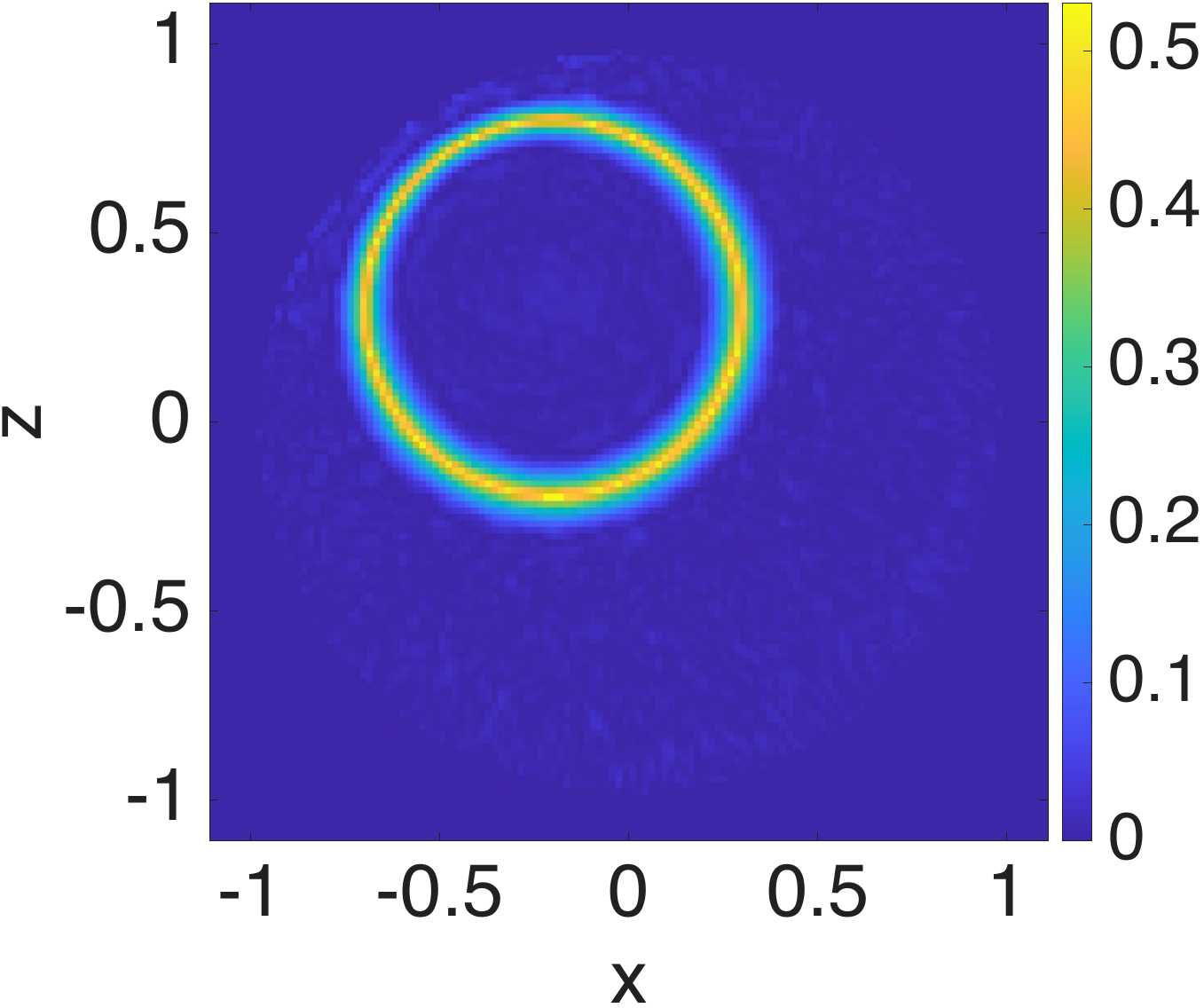}
        \caption{}
        \label{xy-phantom}
    \end{subfigure}
    \begin{subfigure}{0.3\textwidth}
        \centering
        \includegraphics[width=\linewidth]{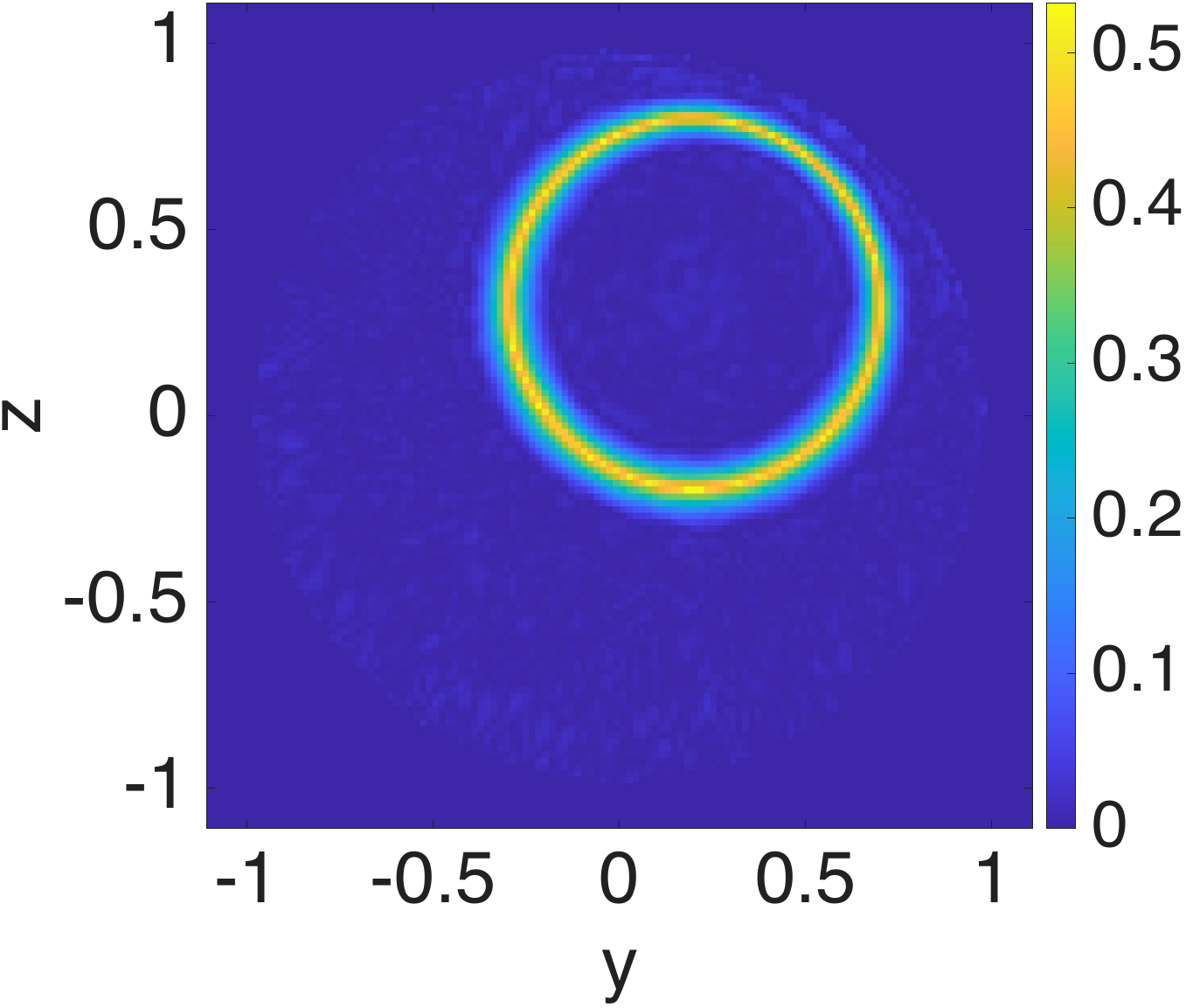}
        \caption{}
        \label{xy-phantom}
    \end{subfigure}
     \caption{
Top views of the cross-sectional slices through the center of the ball phantom. 
Columns correspond to the cross sections by the planes \(z=0.3\), \(y=0.2\), and \(x=-0.2\), respectively. 
The first row (a-c) shows the slices of the exact phantom. 
The second row (d-f) shows the slices of the reconstruction obtained from noiseless cylindrical Radon transform data, while the third row (g-i) shows the absolute error between the phantom and reconstruction.}
\label{fig: clean slices}
\end{figure}
Figure \ref{fig: clean slices} shows the top views of cross-sections of the phantom, reconstruction shown in Figure \ref{fig:clean surface plots}, and absolute difference between them. The absolute difference plots illustrate that the largest discrepancies occur near the jump discontinuity at the boundary of the ball phantom.
\begin{figure}[htbp]
    \centering
    \begin{subfigure}{0.32\textwidth}
        \centering
        \includegraphics[width=\linewidth]{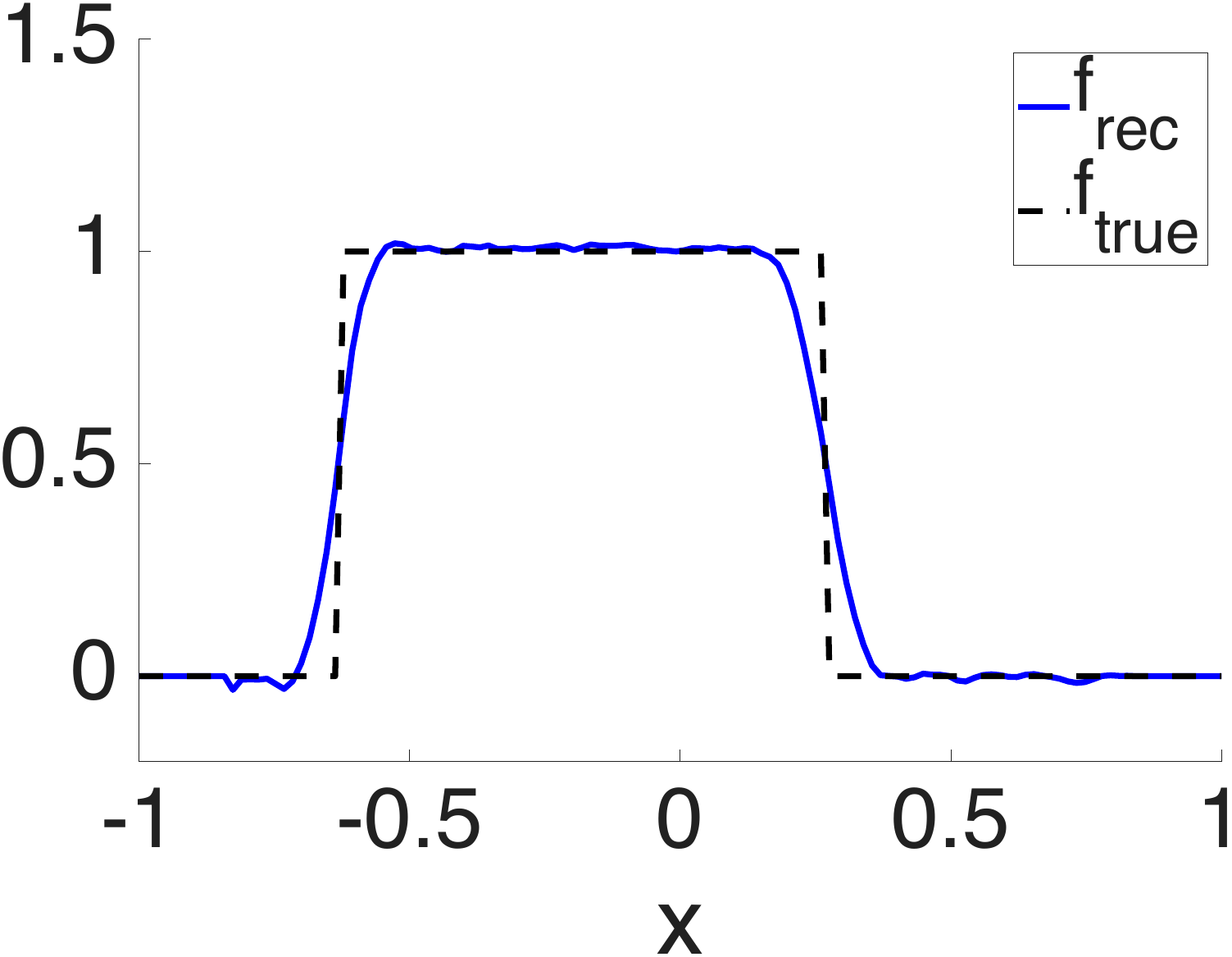}
        \caption{}
        \label{clean-x-profile}
    \end{subfigure}
    \begin{subfigure}{0.32\textwidth}
        \centering
        \includegraphics[width=\linewidth]{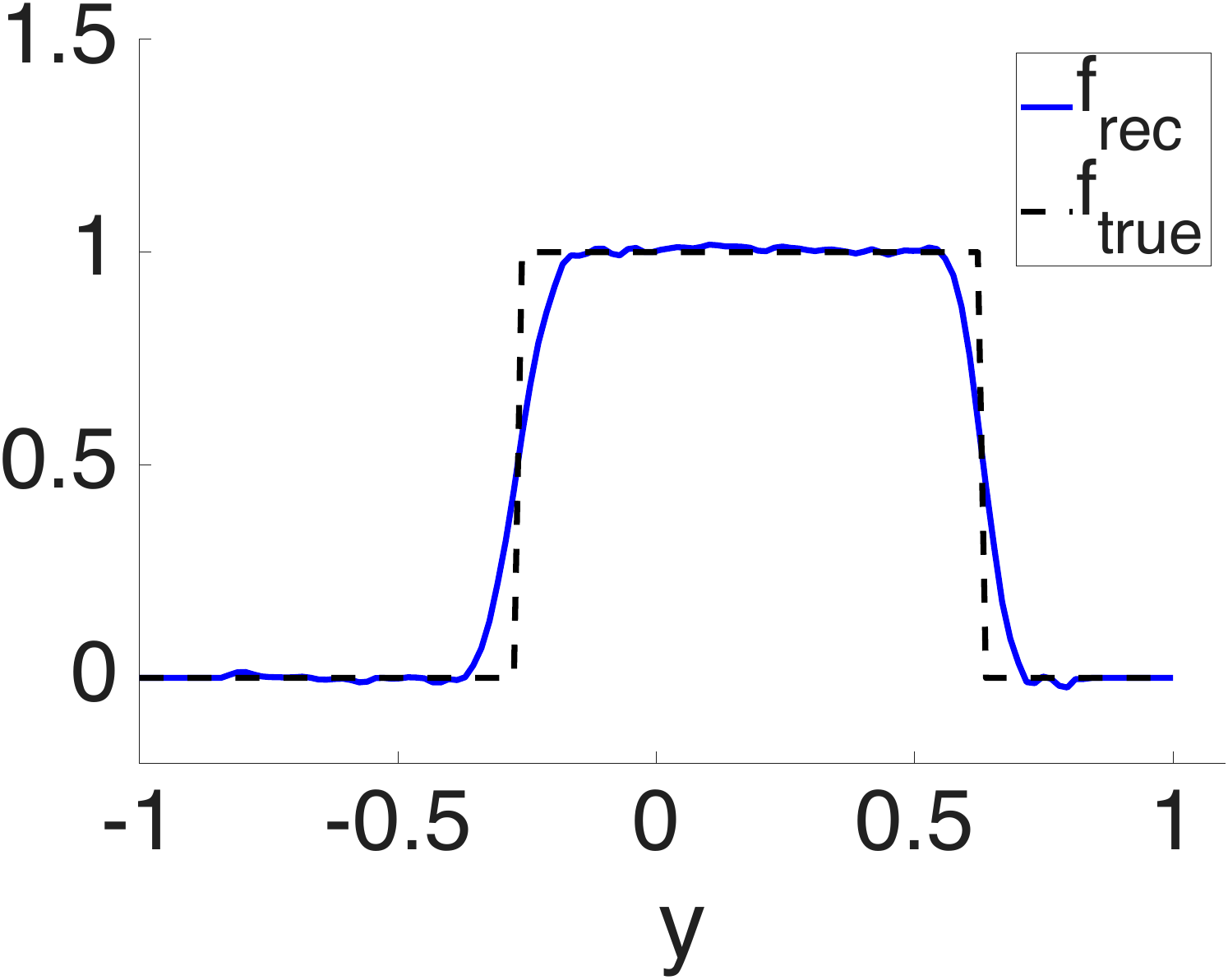}
        \caption{}
        \label{clean-y-profile}
    \end{subfigure}
    \begin{subfigure}{0.32\textwidth}
        \centering
        \includegraphics[width=\linewidth]{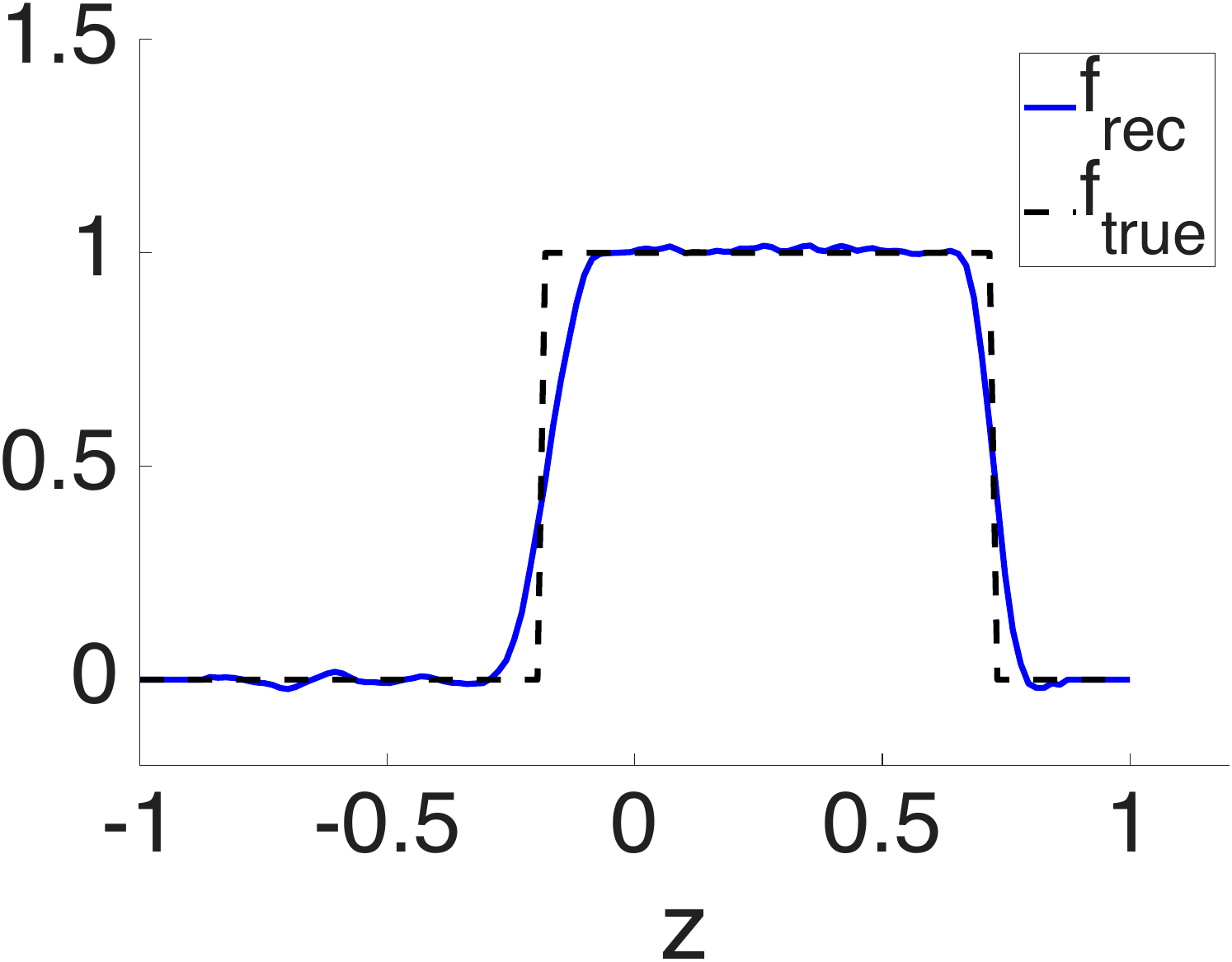}
        \caption{}
        \label{clean-z-profile}
    \end{subfigure}
    \caption{One-dimensional line profiles through the center of the phantom comparing the exact phantom \(f_{\mathrm{true}}\) and the reconstruction obtained from noiseless cylindrical Radon transform data. Comparison of
(a) \(x\)-profiles at \((y,z)=(0.2,0.3)\), 
(b) \(y\)-profiles at \((x,z)=(-0.2,0.3)\), 
(c) \(z\)-profiles at \((x,y)=(-0.2,0.2)\).}
    \label{fig:clean profile plots}
 \end{figure}

The one-dimensional profile plots can be seen in Figure \ref{fig:clean profile plots} which are obtained by slicing the images in Figure \ref{fig: clean slices} through the center of the phantom. 

For the reconstruction from the noiseless data, the relative maximum error was \(0.5491\)\footnote{The relative maximum error was computed as the ratio of the maximum pointwise error between the exact and numerically computed values to the maximum absolute value of the exact data.}. The relative \(L^p\)-error is found by numerically approximating 
\begin{equation}\label{L^2 error R^3}
L_{rel}^p(f_{rec},f_{true})=\Bigg(\frac{\int_{\mathbb{R}^3} |f_{rec}(x,y,z)-f_{true}(x,y,z)|^p\, dV}{\int_{\mathbb{R}^3} |f_{true}(x,y,z)|^p\, dV}\Bigg)^{1/p}.\end{equation} The relative \(L^2\)- and \(L^1\)-errors were measured to be \(0.2635\) and \(0.2664\), respectively. Lastly, the center of mass error was \(0.0017\). The small center of mass error indicates that the reconstruction accurately captures the overall shape. In contrast, the large relative maximum error reflects the inaccuracies near the singularity of the ball phantom.

\subsection{Reconstruction from Noisy Data}
Having quantified the discretization error by applying the inversion algorithm to the noiseless forward data, we now apply the reconstruction procedure to noisy cylindrical Radon transform data to investigate the stability of the inversion algorithm. Additive white Gaussian noise with SNR \(20\) dB was added to the full cylindrical Radon transform data array prior to inversion. In the noisy setting, the weighted integral in the radial variable \eqref{eq:G} becomes highly sensitive to perturbations for small values of r because of the singular factor \(1/r\). Although the reconstruction from noiseless cylindrical Radon transform data remained stable when sampling radii near \(r=0\), the addition of noise caused substantial amplification. To reduce this instability, we regularized the singular weight near \(r=0\) by replacing \(2/r\) with \(2/(r+\epsilon)\) on a small interval of radii. Specifically, we used
\[
G_\epsilon(v,p)
=
\int_{r_{\min}}^{0.05} Cf(v,p,r)\frac{2}{r+\epsilon}\,dr
+
\int_{0.05}^{2} Cf(v,p,r)\frac{2}{r}\,dr,
\]
where \(r_{\min}=0.004\) is the smallest sampled radius. The parameter \(\epsilon\) was selected empirically by comparing \(G_\epsilon(v,p)\), computed from noisy data, with the corresponding noiseless quantity \(G(v,p)\). Among logarithmically spaced values \(\epsilon\in(10^{-4},10^{-1})\), the relative \(L^2\)-error was minimized near \(\epsilon=0.0055\), which was used in the noisy data reconstructions. This choice of \(\epsilon\) is used only as a benchmark in this controlled numerical experiment, since the noiseless quantity \(G(v,p)\) is not available in practical measured-data settings.

\begin{figure}[htbp]
\centering
\includegraphics[width=0.5\linewidth]{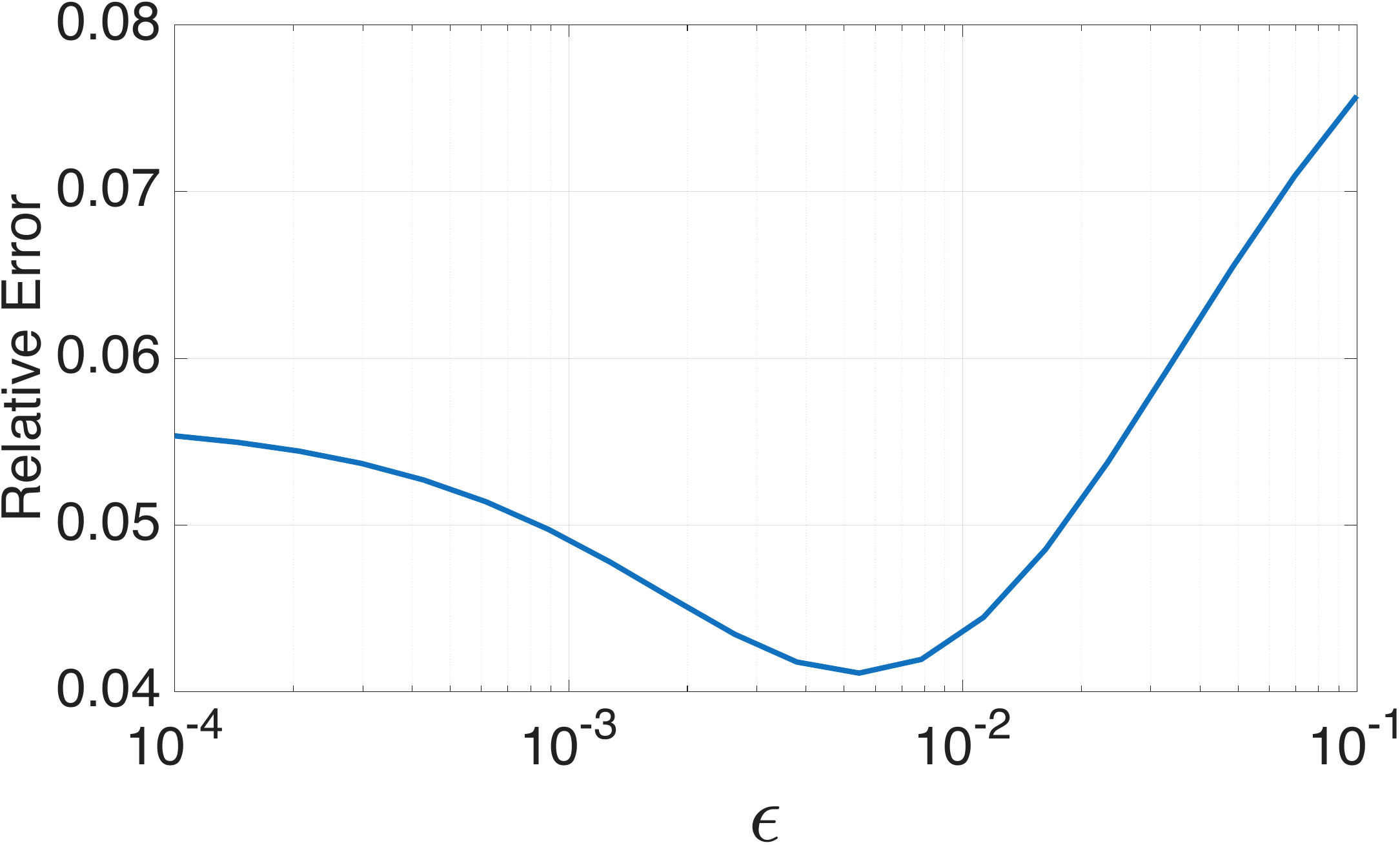}
    \caption{Relative \(L^2\)-error between \(G_\epsilon(v,p)\) and  \(G(v,p)\) as \(\epsilon\) varies.}
    \label{fig:reg param}
\end{figure}
Apart from the modified weighting near \(r=0\), all remaining inversion steps were identical to those used in the noiseless-data reconstruction. 

Figure \ref{fig:noisy surface plots} shows the surface plots where one coordinate is fixed at the center of the phantom. These plots look very similar to Figure \ref{fig:clean surface plots} suggesting that the reconstruction remains stable for this noise level.
\begin{figure}[htbp]
\centering
    \begin{subfigure}{0.32\textwidth}
        \centering
        \includegraphics[width=\linewidth]{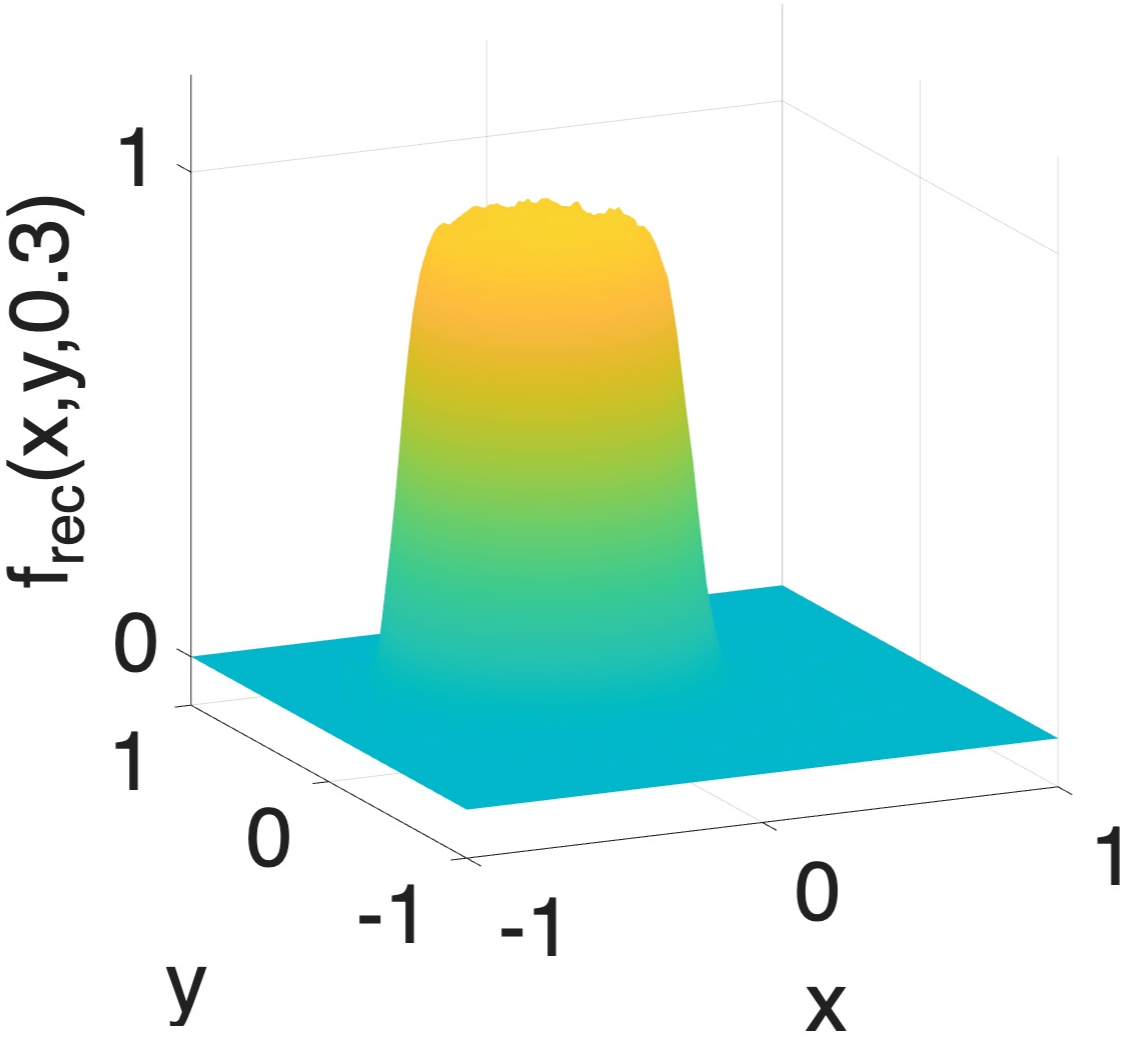}
        \caption{}
        \label{3D-yz-phantom1}
    \end{subfigure}
    \begin{subfigure}{0.32\textwidth}
        \centering
        \includegraphics[width=\linewidth]{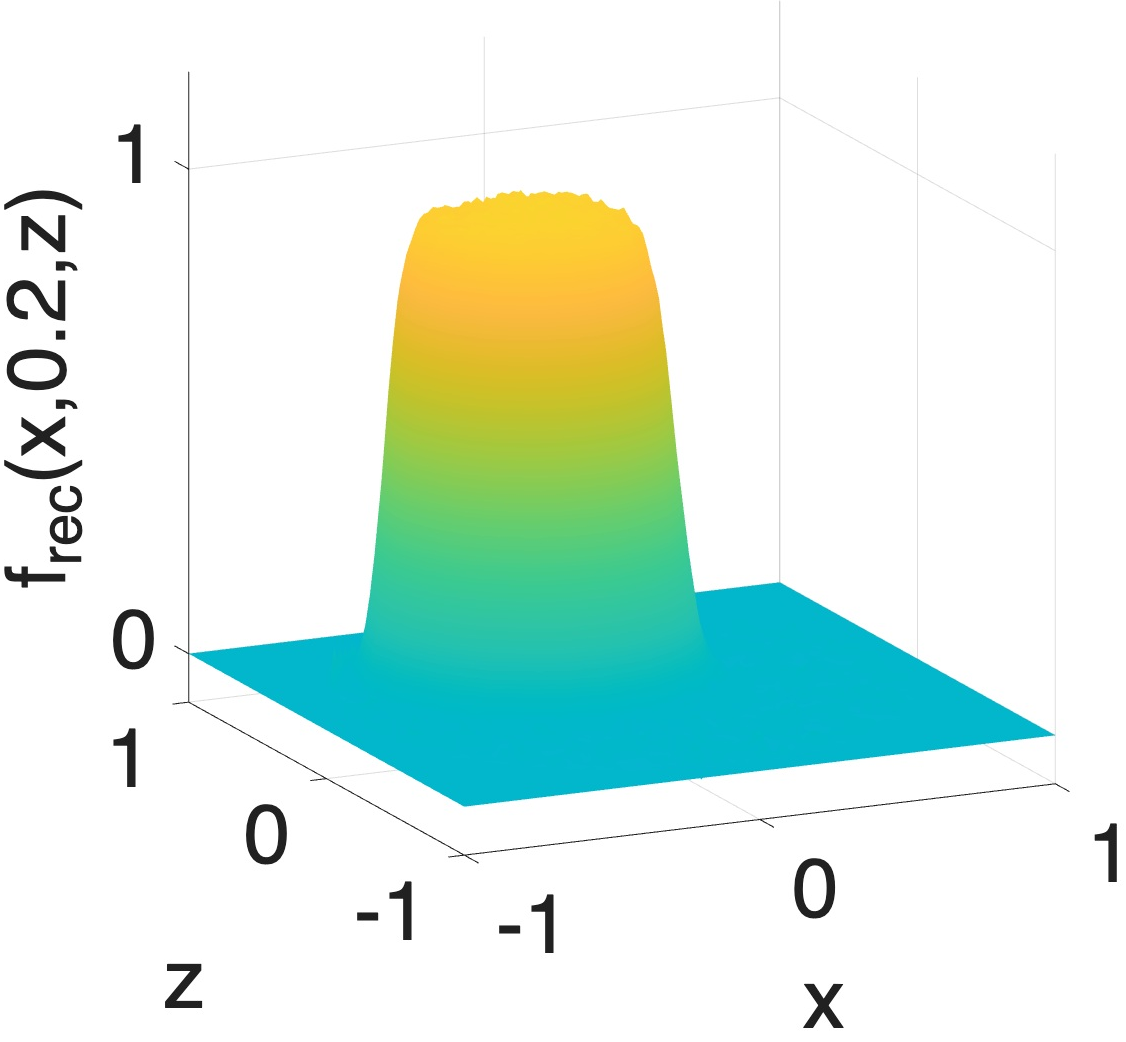}
        \caption{}
    \end{subfigure}
    \begin{subfigure}{0.32\textwidth}
        \centering
        \includegraphics[width=\linewidth]{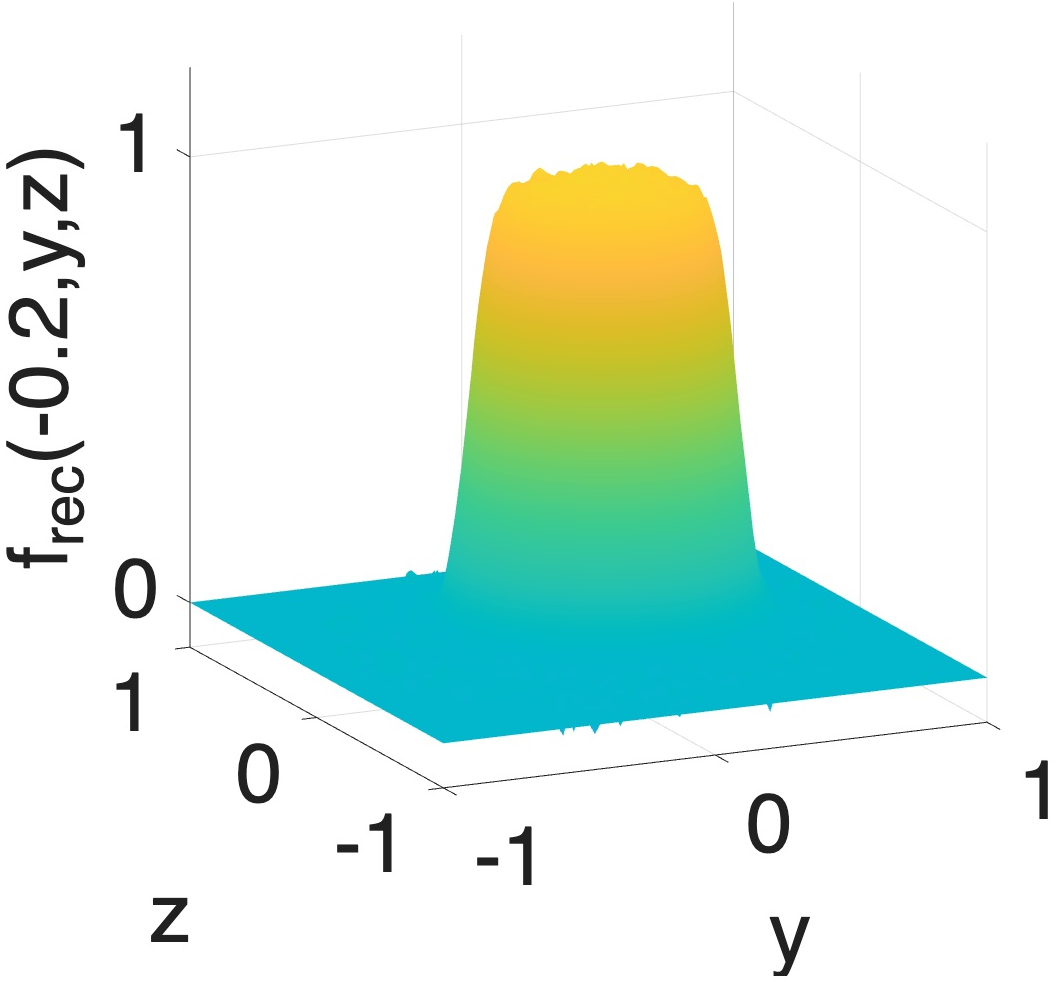}
        \caption{}
    \end{subfigure}

    \caption{The surface plots of the reconstruction from noisy cylindrical Radon transform data using cross sections through the center of the phantom. (a) \(f_{rec}(x,y,0.3)\), (b) \(f_{rec}(x,0.2,z)\), (c) \(f_{rec}(-0.2,y,z)\).}
    \label{fig:noisy surface plots}
\end{figure}

\begin{figure}[htbp]
    \begin{subfigure}{0.32\textwidth}
        \centering
        \includegraphics[width=\linewidth]{Images/Phantom_slices/phantom-xy.pdf}
        \caption{}
        \label{phantom-xy-slice}
    \end{subfigure}
    \begin{subfigure}{0.32\textwidth}
        \centering
        \includegraphics[width=\linewidth]{Images/Phantom_slices/phantom-xz.pdf}
        \caption{}
        \label{phantom-xz-slice}
    \end{subfigure}
    \begin{subfigure}{0.32\textwidth}
        \centering
        \includegraphics[width=\linewidth]{Images/Phantom_slices/phantom-yz.pdf}
        \caption{}
        \label{phantom-yz-slice}
    \end{subfigure}\\
    \begin{subfigure}{0.32\textwidth}
        \centering
        \includegraphics[width=\linewidth]{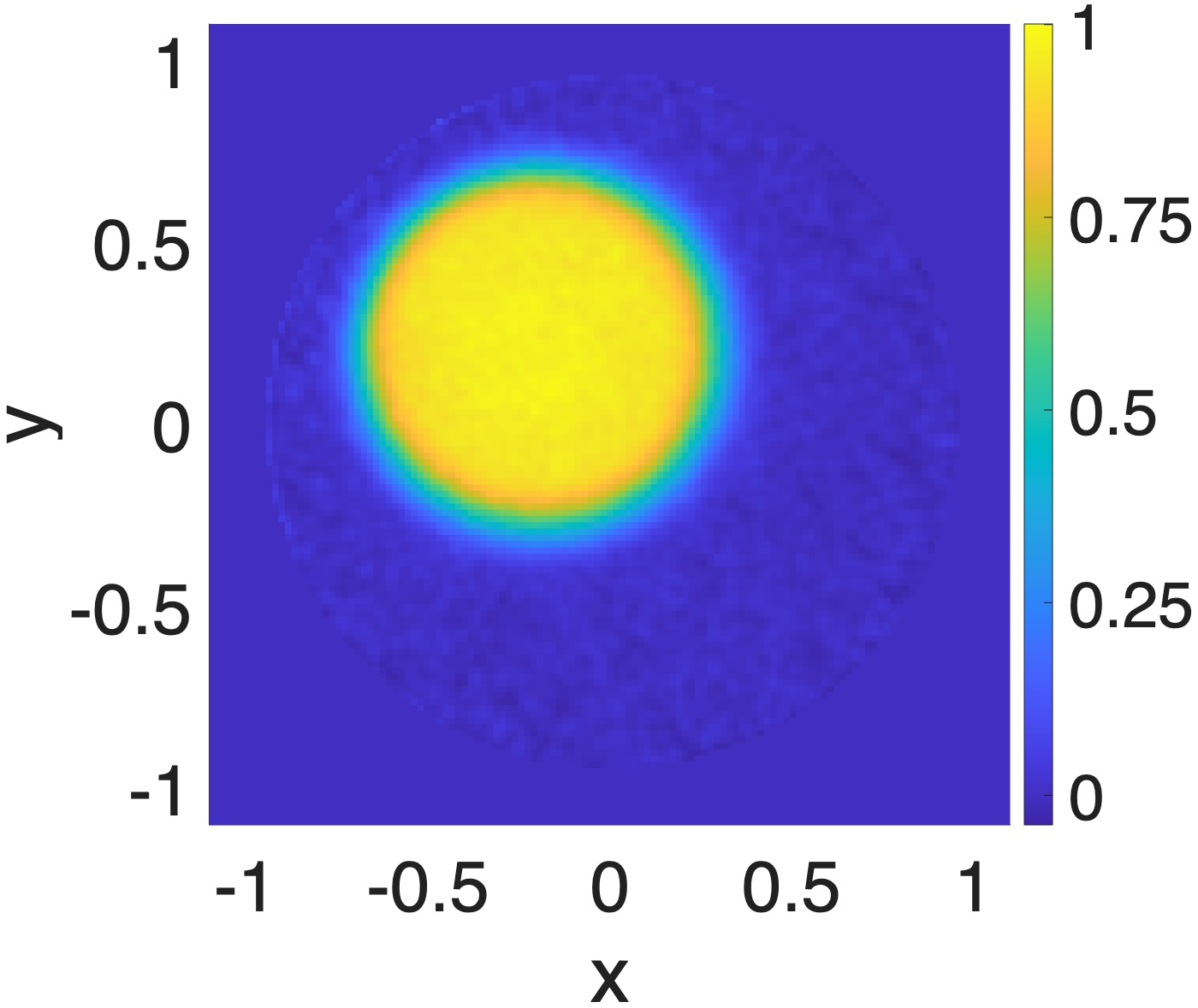}
        \caption{}
        \label{noisy-xy-slice}
    \end{subfigure}
    \begin{subfigure}{0.32\textwidth}
        \centering
        \includegraphics[width=\linewidth]{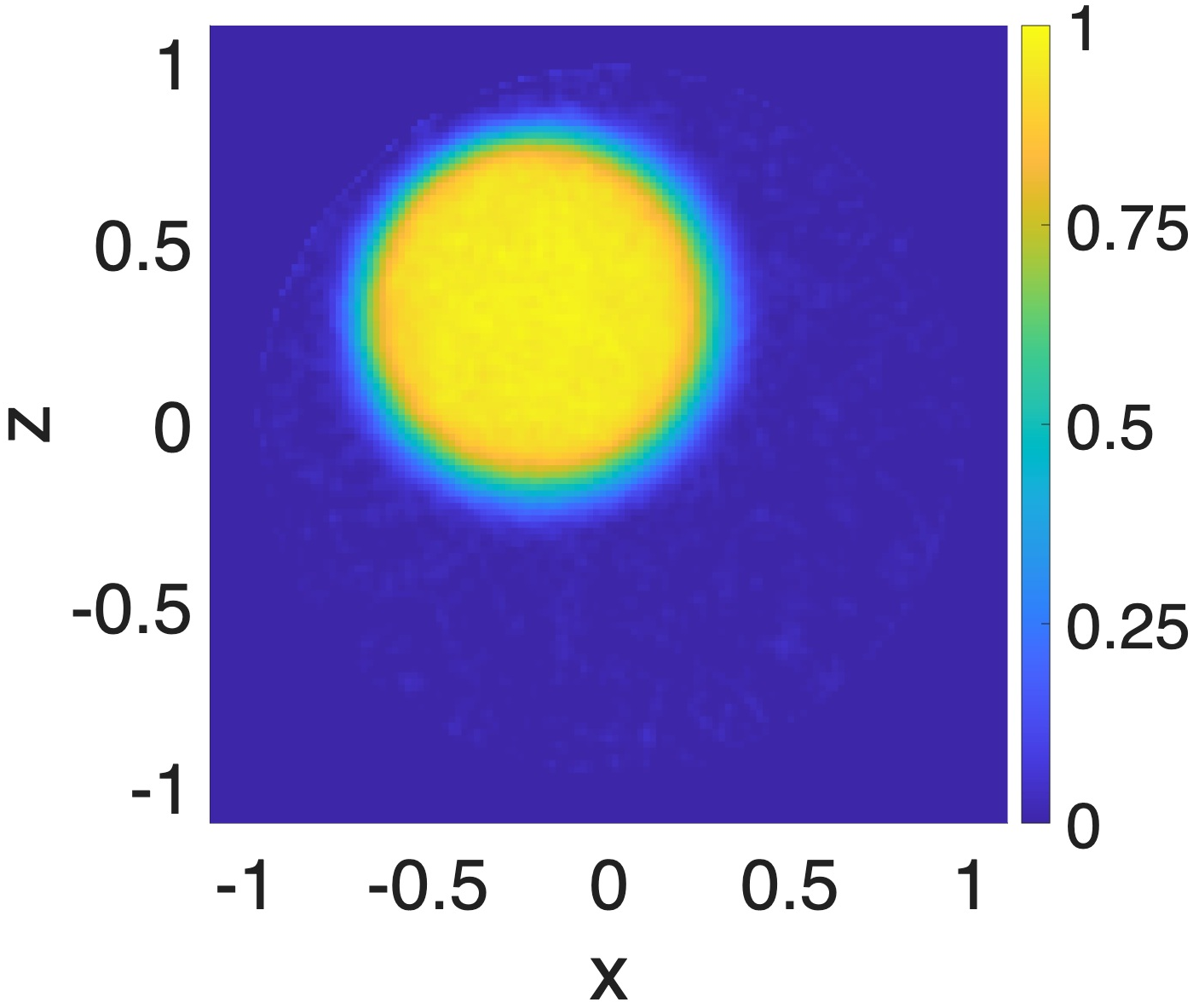}
        \caption{}
        \label{noisy-xz-slice}
    \end{subfigure}
    \begin{subfigure}{0.32\textwidth}
        \centering
        \includegraphics[width=\linewidth]{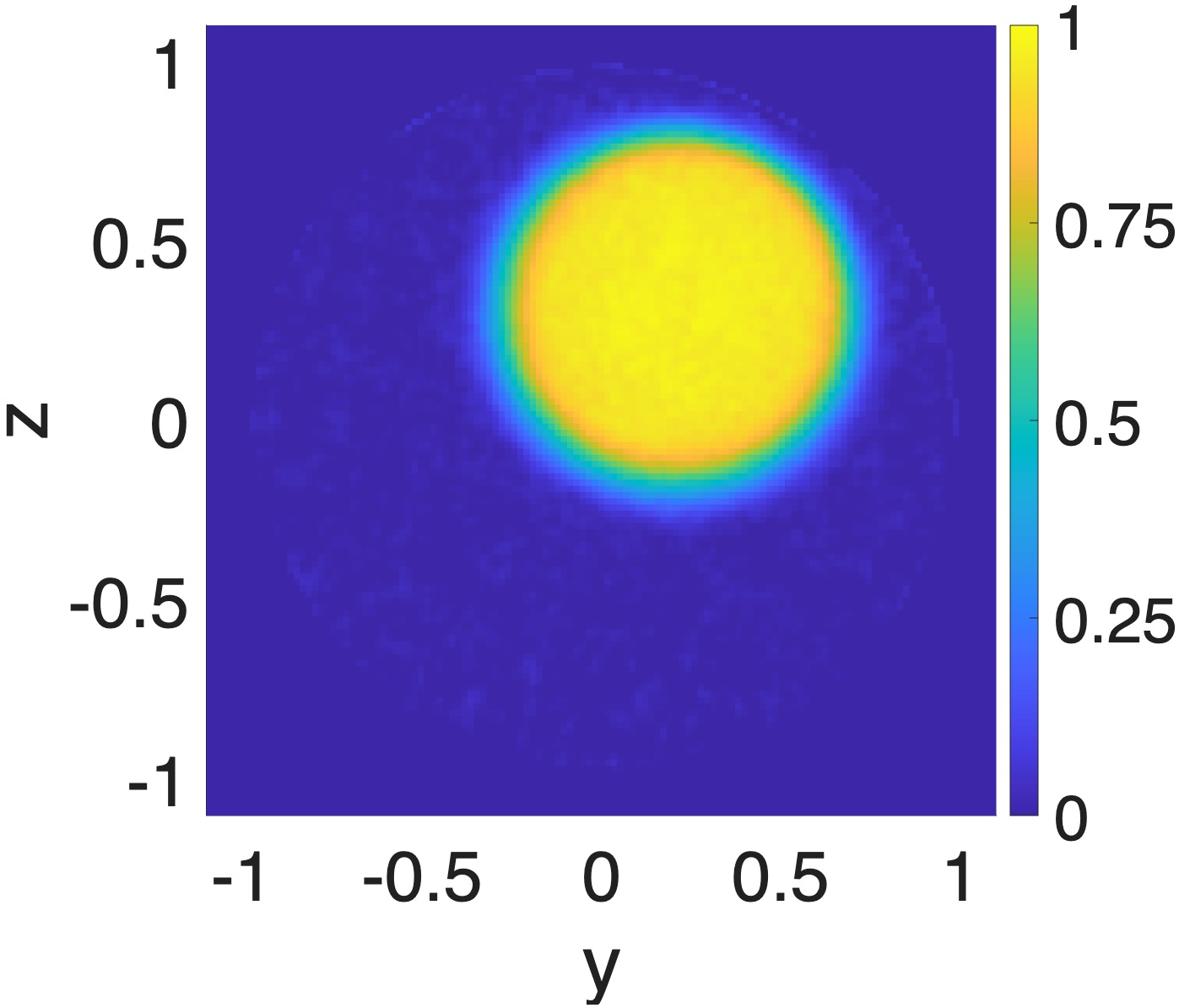}
        \caption{}
        \label{noisy-yz-slice}
    \end{subfigure}\\
    \begin{subfigure}{0.32\textwidth}
        \centering
        \includegraphics[width=\linewidth]{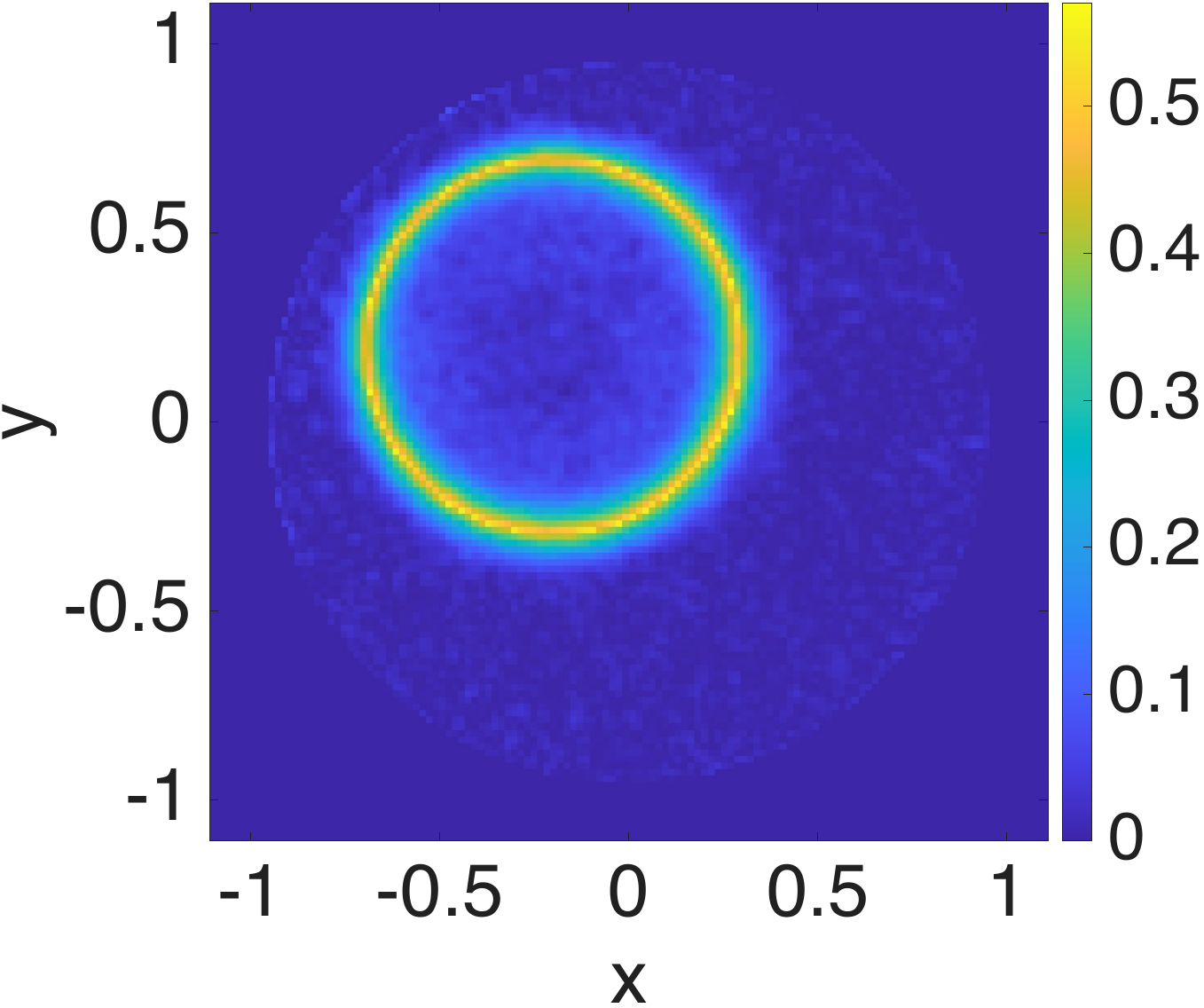}
        \caption{}
        \label{noisy-xy-error}
    \end{subfigure}
    \begin{subfigure}{0.32\textwidth}
        \centering
        \includegraphics[width=\linewidth]{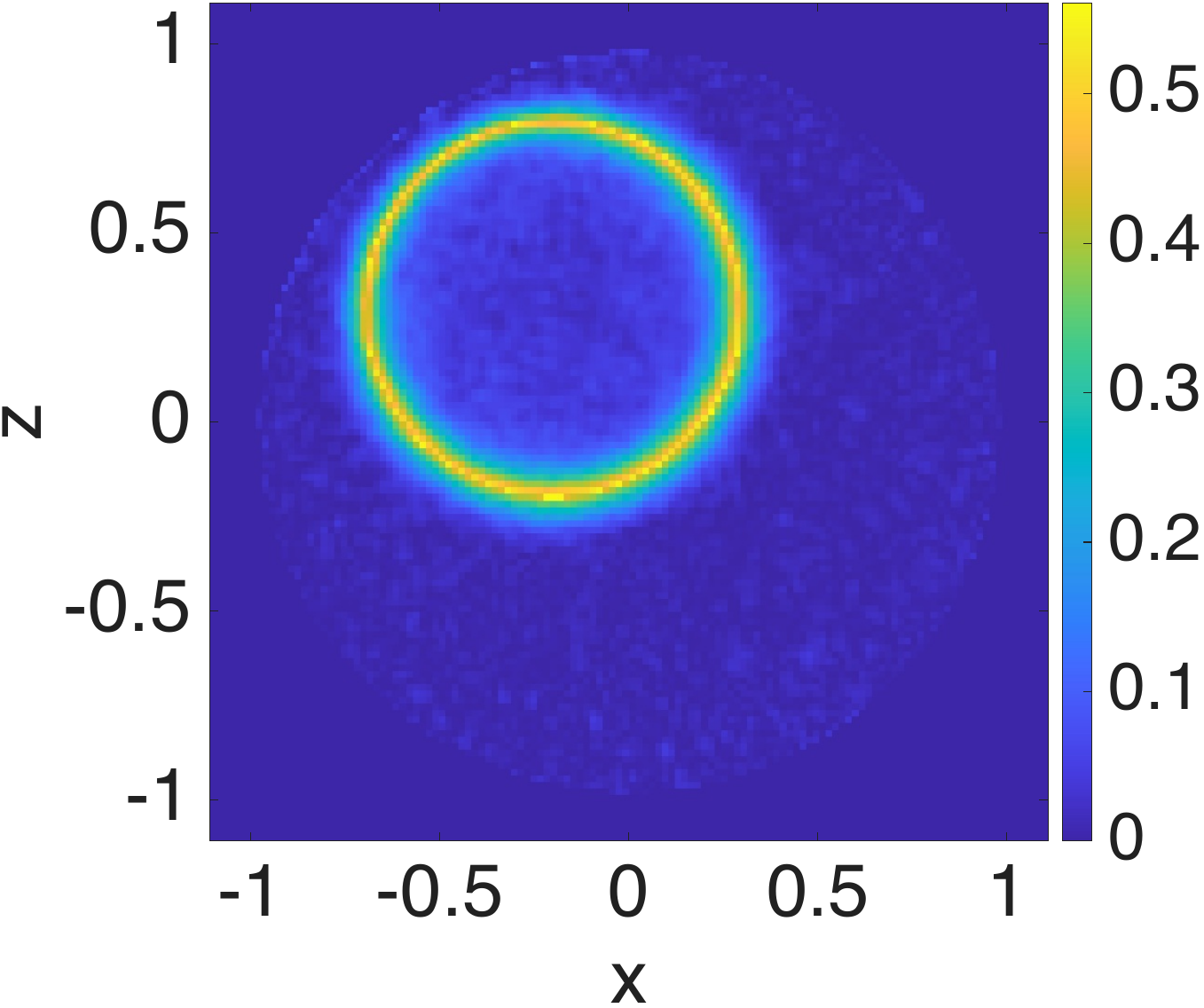}
        \caption{}
        \label{noisy-xz-error}
    \end{subfigure}
    \begin{subfigure}{0.32\textwidth}
        \centering
        \includegraphics[width=\linewidth]{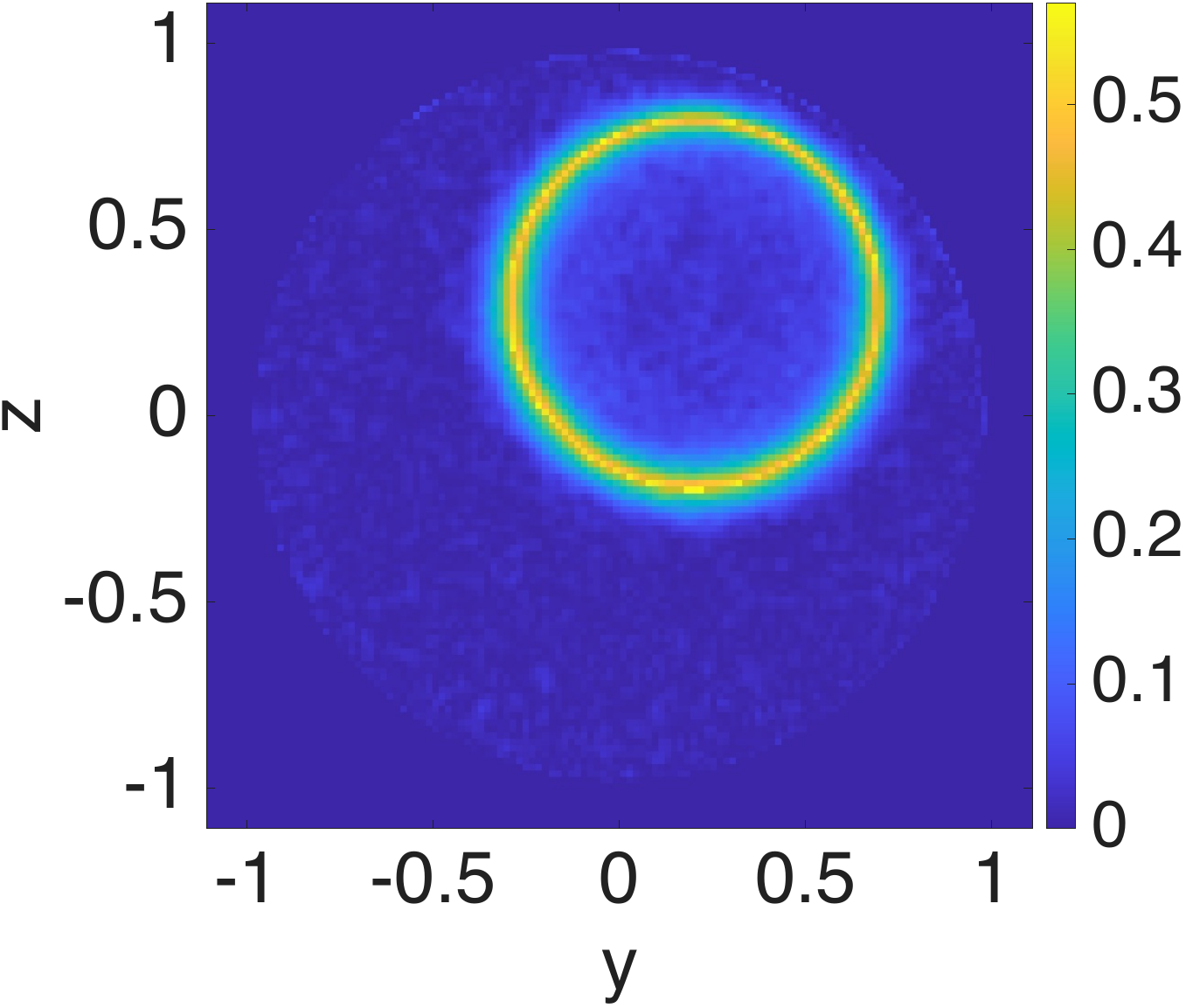}
        \caption{}
        \label{noisy-yz-error}
    \end{subfigure}

    \caption{
Top views of the cross-sectional slices through the center of the ball phantom. 
Columns correspond to the cross sections by the planes \(z=0.3\), \(y=0.2\), and \(x=-0.2\), respectively. 
The first row (a-c) shows the slices of the exact phantom. 
The second row (d-f) shows the slices of the reconstruction obtained from noisy cylindrical Radon transform data, while the third row (g-i) shows the absolute error between the phantom and reconstruction.}
    \label{fig: noisy slices}
\end{figure}

In Figure \ref{fig: noisy slices}, the plots show two-dimensional cross-sections of the reconstruction and absolute error in each coordinate direction for the noisy measurement case. The absolute error plots again illustrate that the largest discrepancies occur near the jump discontinuity at the boundary of the ball phantom.  

\begin{figure}[htbp]
    \begin{subfigure}{0.32\textwidth}
        \centering
        \includegraphics[width=\linewidth]{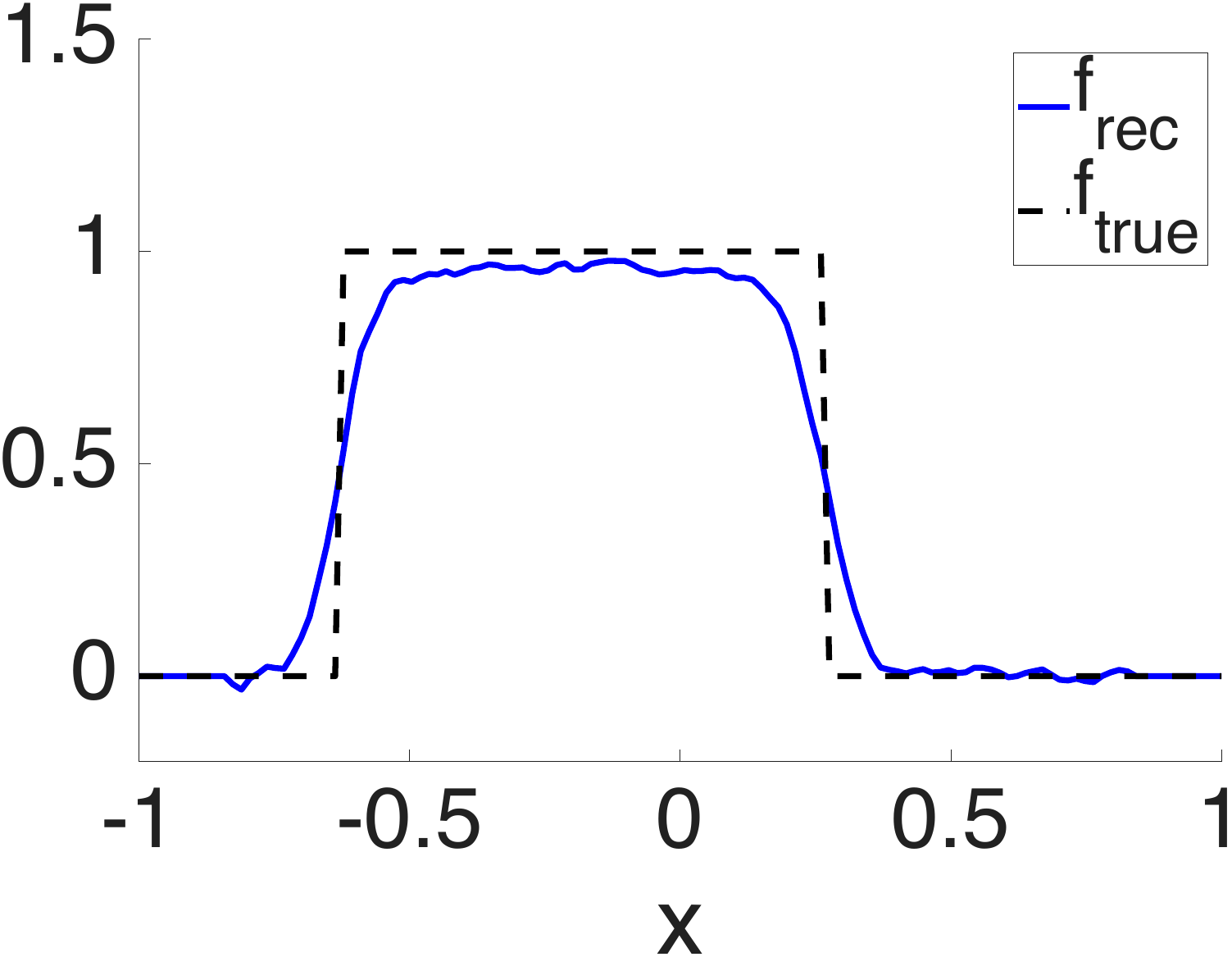}
        \caption{x-profile}
        \label{noisy-x-profile}
    \end{subfigure}
    \begin{subfigure}{0.32\textwidth}
        \centering
        \includegraphics[width=\linewidth]{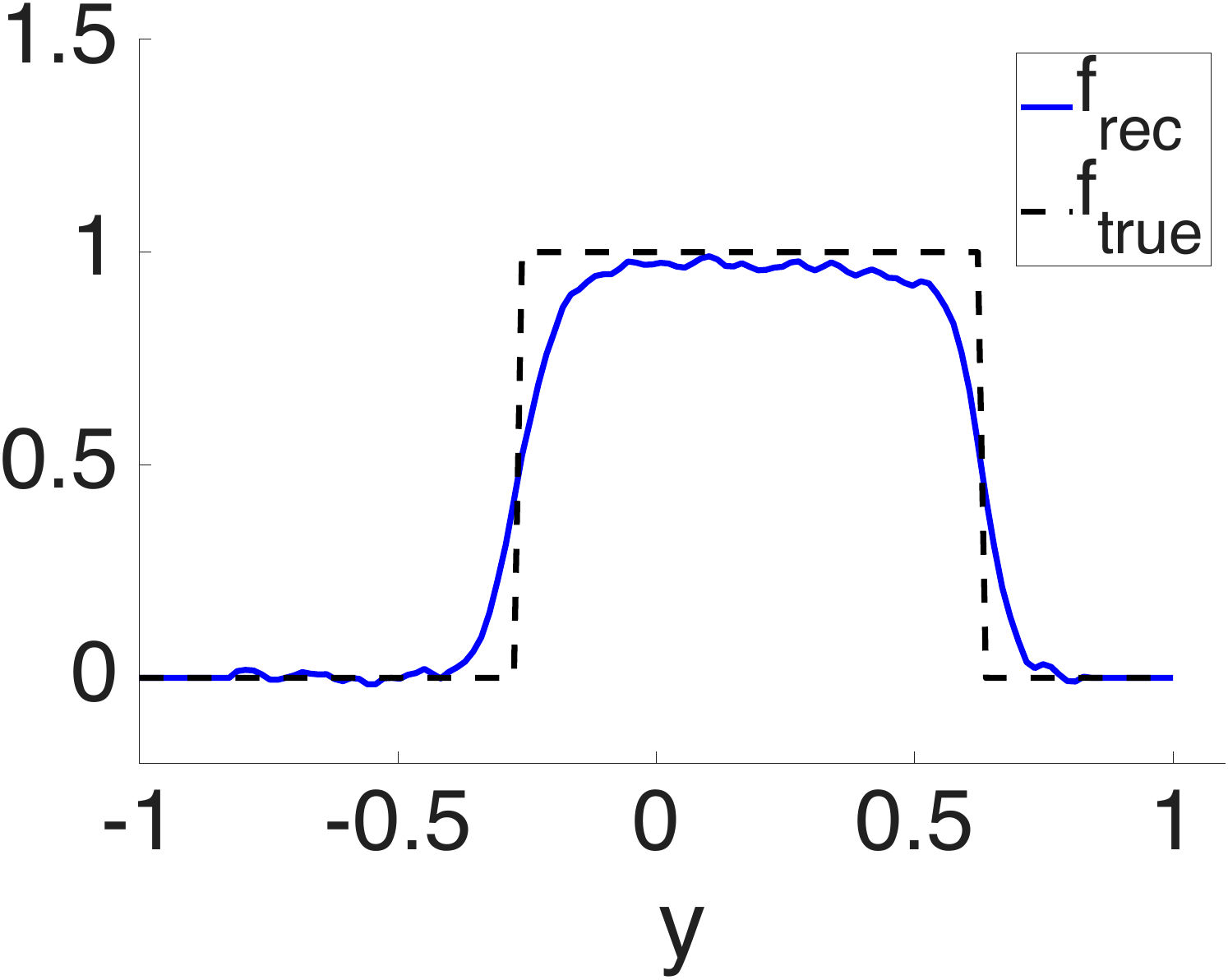}
        \caption{y-profile}
        \label{noisy-y-profile}
    \end{subfigure}
    \begin{subfigure}{0.32\textwidth}
        \centering
        \includegraphics[width=\linewidth]{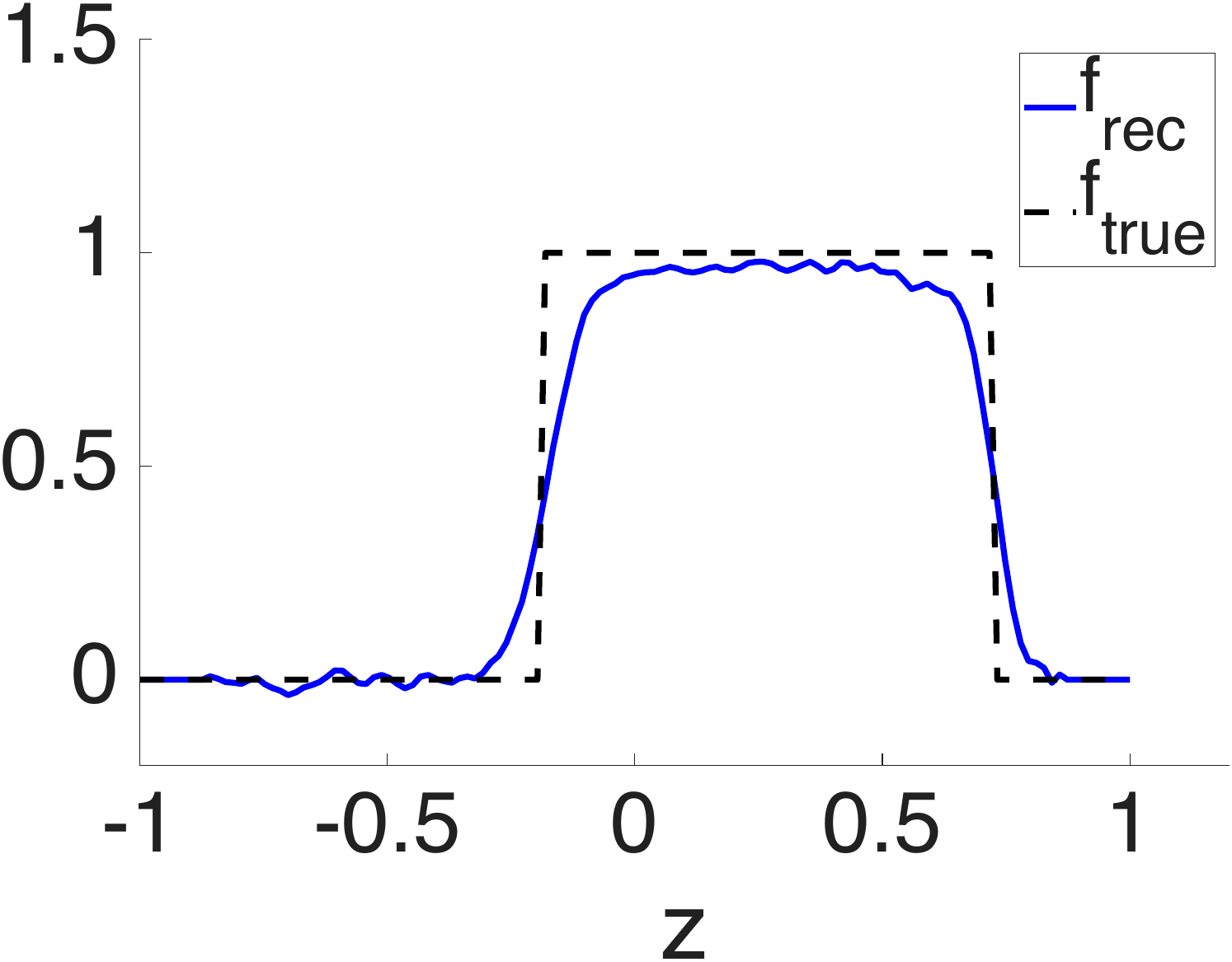}
        \caption{z-profile}
        \label{noisy-z-profile}
    \end{subfigure}

    \caption{One-dimensional line profiles through the center of the phantom comparing the exact phantom \(f_{\mathrm{true}}\) and the reconstruction obtained from noisy cylindrical Radon transform data (SNR \(20\) dB). Comparison of
(a) \(x\)-profiles at \((y,z)=(0.2,0.3)\). 
(b) \(y\)-profiles at \((x,z)=(-0.2,0.3)\). 
(c) \(z\)-profiles at \((x,y)=(-0.2,0.2)\).}
    \label{fig:noisy profile plots}
\end{figure}

Lastly, Figure \ref{fig:noisy profile plots} shows the one-dimensional profile plots obtained by slicing the images in Figure \ref{fig: noisy slices} through the center of the phantom. Again, we see the most discrepancies around the singularity of the phantom.

For the noisy data, we again used \eqref{L^2 error R^3} to find the relative \(L^2\)-error  and \(L^1\)-error which increased to \(0.3054\), and \(0.3802\), respectively. The relative maximum error is \(0.5810\). However, the center of mass error remained quite low at \(0.0044\). This suggests that the dominant errors are again concentrated near the discontinuity of the phantom.

Comparing the error from the noiseless and noisy cases in table \ref{tab: errors}, the error in the noisy case is slightly higher. The \(L^1\)-error and \(L^2\)-error increased by \(0.1138\) and \(0.0419\), respectively. The relative maximum error increased by \(0.0319\). Lastly, the center of mass error increased by \(0.0027\). Overall, noise does lead to a measurable degradation in reconstruction quality. However, the reconstruction remains stable as global geometric properties are preserved, and large errors are localized near the phantom singularities. 

\begin{table}[H]
\centering
\begin{tabular}{c|cc}
\hline
 Error Type & Noiseless CRT Data & Noisy CRT Data \\
\hline
Relative \(L^2\) & 0.2635 & 0.3054  \\
Relative \(L^1\) & 0.2664 & 0.3802  \\
Relative Maximum & 0.5491 & 0.5810 \\
Center of Mass & 0.0017 & 0.0044  \\
\hline
\end{tabular}
\caption{Comparison of reconstruction errors obtained from noiseless and noisy cylindrical Radon transform data.}
\label{tab: errors}
\end{table}

\section{Conclusion}\label{conclusion}
In this paper, we studied the analytic inversion of the cylindrical Radon transform corresponding to an overdetermined set of cylindrical surfaces in \(\mathbb{R}^n\). We analyzed its geometric properties and derived an integral identity relating the cylindrical Radon transform to the Radon and Funk transforms. This identity led to an analytical inversion method for the cylindrical Radon transform via the known Funk and Radon transform inversions. 

We also presented numerical experiments implementing the inversion formula for a ball phantom. The results demonstrate that the proposed algorithm is stable under discretization effects and moderate noise. However, noticeable reconstruction errors occur near the singularity of the phantom, indicating that the inversion algorithm tends to smooth sharp features near singularities.

In future work, we plan to incorporate edge-preserving techniques into the numerical algorithm to improve reconstruction accuracy near singularities. We are also interested in studying inversion from other families of cylinders that are relevant to photoacoustic tomography and yield complete data sets for reconstruction.

\section{Appendix}\label{appendix}
\textbf{Proof of Equation \eqref{eq:Rfball3D}:}\label{app:proof of 3D ball}
By translation invariance of the Radon transform \eqref{R(T_pf)=T_pRf},
\[
R\mathcal{X}_{B(c,t)}(\omega,s)
=
RT_{-c}\mathcal{X}_{B(0,t)}(\omega,s)
=
R\mathcal{X}_{B(0,t)}(\omega,s-\omega\cdot c).
\]
Let \(A\in SO(3)\) be a rotation such that \(A\omega=e_3\). Since $\mathcal{X}_{B(0,t)}$ is rotationally invariant, by equation \eqref{M_ARf=RM_Af}, we have
\begin{equation}\label{Rball=Rballorigin}
R\mathcal{X}_{B(c,t)}(\omega,s)
=
R\mathcal{X}_{B(0,t)}(e_3,s-\omega\cdot c).
\end{equation}
Using the definition of the Radon transform \eqref{Radon transform},
\[
R\mathcal{X}_{B(0,t)}(e_3,s-\omega\cdot c)
=\int_{\mathbb{R}^3}\mathcal{X}_{B(0,t)}(y)\delta(y\cdot e_3-(s-\omega\cdot c))\, dy=
\int_{|y|<t}
\delta(y\cdot e_3-(s-\omega\cdot c))\,dy.
\]
Writing \(y=(\bar y,y_3)\), where \(\bar y=(y_1,y_2)\), and setting
\[
\sigma=s-\omega\cdot c,
\]
gives
\[
R\mathcal{X}_{B(0,t)}(e_3,\sigma)
=
\int_{\mathbb R^2}\int_{\mathbb R}
\mathcal{X}_{B(0,t)}(\bar y,y_3)
\delta(y_3-\sigma)\,dy_3\,d\bar y.
\]
Integrating with respect to \(y_3\) yields
\[
R\mathcal{X}_{B(0,t)}(e_3,\sigma)
=
\int_{\mathbb R^2}
\mathcal{X}_{\{|\bar y|^2+\sigma^2\le t^2\}}(\bar y)
\,d\bar y.
\]
The integrand is the indicator function of a disk in \(\mathbb R^2\) with radius
\(\sqrt{t^2-\sigma^2}\) when \(|\sigma|<t\), and it is zero otherwise. Therefore,
\[
R\mathcal{X}_{B(0,t)}(e_3,\sigma)
=
\begin{cases}
\pi(t^2-\sigma^2), & |\sigma|<t,\\[2mm]
0, & \text{otherwise}.
\end{cases}
\]
Substituting \(\sigma=s-\omega\cdot c\) and recalling the equality in equation \eqref{Rball=Rballorigin} gives
\[
R\mathcal{X}_{B(c,t)}(\omega,s)
=
\begin{cases}
\pi\bigl(t^2-(s-\omega\cdot c)^2\bigr),
& |s-\omega\cdot c|<t,\\[2mm]
0,
& \text{otherwise},
\end{cases}
\]
which proves \eqref{eq:Rfball3D}.
\qed

\section*{Acknowledgements}
This work was partially supported by NSF DMS grant 2206279.\\

\bibliographystyle{siam}
\bibliography{Paperdraft}

\end{document}